\documentclass[preprint, 12pt]{article}
\label{key}
\usepackage{calc}
\usepackage[all]{xy}
\usepackage[centertags]{amsmath}
\usepackage{latexsym}
\usepackage{amsfonts}
\usepackage{amssymb}
\usepackage{amsthm}
\usepackage{geometry}
\usepackage{fancyhdr}
\usepackage[dvips]{epsfig}
\usepackage{newlfont}
\usepackage{fancyhdr}
\usepackage{color}
\usepackage{blindtext}
\usepackage{graphics}
\usepackage{newlfont}
\usepackage{wrapfig}
\usepackage{floatflt}
\usepackage{floatflt}
\usepackage{pifont}
\usepackage{hyperref}
\usepackage[english]{babel}
\usepackage[latin1]{inputenc}
\usepackage{graphicx}
\usepackage{t1enc}
\usepackage{pifont}
\usepackage{makeidx}
\usepackage{pgf,tikz,pgfplots}
\pgfplotsset{compat=1.13}
\usepackage{mathrsfs}
\usetikzlibrary{arrows}
\usepackage{tikz-cd}
\usepackage{enumitem}
\usepackage{kantlipsum}
\usepackage{authblk}
\usetikzlibrary{arrows}
\fancypagestyle{plain}{ }

\newlength{\defbaselineskip}
\newtheorem{definition}{Definition}[section]
\newtheorem{proposition}[definition]{Proposition}
\newtheorem{remark}[definition]{Remark}
\newtheorem*{Key words}{Key words}

\newtheorem{example}[definition]{Example}

\newtheorem*{Proof of Theorem Equivalence between realization functors}{Proof of Theorem \ref{Equivalence_between_realization_functors}}
\newtheorem{Proof of proposition Lif adjunction symmetric sequences}{Proof of proposition \ref{lift_adjonction_symmetric_sequence}}
\newtheorem*{Proof of Theorem Cofree coalgebras}{Proof of Theorem \ref{Cofree Cocommutative_Coalgebra}}
\newtheorem*{Proof of main theorem} {Proof of theorem \ref{Quillen_equivalence_operads}} 
\newtheorem{theorem}[definition]{Theorem}
\newtheorem*{Proof of Proposition CoFree Cosimplicial  Cocommutative Coalgebra} {Proof of Proposition \ref{coFree_Cocomutative_Coalgebra}} 
\newtheorem{corollary}[definition]{Corollary}

\newtheorem{mylemma}[definition]{Lemma}
\newtheorem*{Abstract}{Abstract}
\newtheorem*{Theorem A}{Theorem A}
\newtheorem*{Theorem B}{Theorem B}
\newtheorem*{Proposition A-1}{Proposition A-1}
\newtheorem*{Acknowledgements}{Acknowledgements}
\newtheorem*{Corollary A-1}{Corollary A-1}
\newtheorem*{Corollary A-2}{Corollary A-2}
\newtheorem*{Proof of Proposition Dn approximation finite cell} {Proof of Proposition \ref{Dn_approximation_theorem}}  
\newtheorem*{Proof of proposition cosimplicial resolution cell functor} {Proof of proposition \ref{Proposition Cosimplicial Resolution presented cell functor}} 
\newtheorem*{Proof of proposition Pro right module structure derivative} {Proof of proposition \ref{pro_right_module}} 
\newtheorem*{Proof of thm new model for derivatives Algebras} {Proof of theorem \ref{model_goowillie_derivative_functor_from_algebra}} 
\newtheorem*{Proof of proposition Restriction natural transformation to finite proposition} {Proof of Proposition \ref{Restriction natural transformation to finite proposition}} 
\newtheorem*{Proof of proposition Nat}{Proof of proposition \ref{Natural_Transformation_Exponential}}
\newtheorem*{Proof of theorem C_Delta}{ Proof of theorem \ref{coFree_Cocomutative_Coalgebra}}
\newtheorem*{Proposition D}{ Proposition \ref{Arone_ching_decomposition_of_functor}}
\newtheorem*{Proof of proposition}{Proof of proposition \ref{Right_comodule_derivatives}}
\newtheorem*{Proof of theorem}{Proof of theorem \ref{Nick_Kuhn}}
\newtheorem*{Proof of corollary}{Proof of corollary  \ref{Derivatives}}
\newtheorem*{Proof of Lemma derivative representable functor}{Proof of Lemma  \ref{Lemma Computing the derivatives of Representable functor}}
\newtheorem*{Proof of Chain Rule Ch}{Proof of proposition  \ref{Chain_Rule_Ch}}
\newtheorem*{Sketch of the proof} {Sketch of the proof of Proposition \ref{Projective model structure on functors}} 
\newtheorem*{Proof of Tower theorem} {Proof of theorem \ref{Taylor_Tower_Equal_Fake}} 
\newtheorem*{Proof of model derivative ch functors} {Proof of theorem \ref{Model_derivative_ch_functor_theorem}}

\newsavebox{\bparcould}\newlength{\lparcould}

 \theoremstyle{plain}  \makeindex
\makeatletter
\newcommand{\thechapterwords}
{ \ifcase \thechapter\or Premier\or Deux\or Trois\or Quatre\or
	Cinq\or
	Six\or Sept\or Huit\or Neuf\or Dix\or Onze\fi}
\def\thickhrulefill{\leavevmode \leaders \hrule height 1ex \hfill \kern \z@}
\def\@makechapterhead#1{    \vspace*{15\p@}  {\parindent \z@ \centering \reset@font
		\thickhrulefill\quad
		\scshape \@chapapp{} \thechapterwords
		\quad \thickhrulefill
		\par\nobreak
		\vspace*{15\p@}        \interlinepenalty\@M
		\hrule
		\vspace*{15\p@}        \large \bfseries #1\par\nobreak
		\par
		\vspace*{15\p@}        \hrule
		\vskip 60\p@
	}}
	\def\@makeschapterhead#1{    \vspace*{15\p@}  {\parindent \z@ \centering \reset@font
			\thickhrulefill
			\par\nobreak
			\vspace*{15\p@}        \interlinepenalty\@M
			\hrule
			\vspace*{15\p@}        \large \bfseries #1\par\nobreak
			\par
			\vspace*{15\p@}        \hrule
			\vskip 60\p@
		}}
		\def\@makechapterhead#1{    \vspace*{15\p@}  {\parindent \z@ \centering \reset@font
				\thickhrulefill\quad
				\scshape \@chapapp{} \thechapterwords
				\quad \thickhrulefill
				\par\nobreak
				\vspace*{15\p@}        \interlinepenalty\@M
				\hrule
				\vspace*{15\p@}        \large \bfseries #1\par\nobreak
				\par
				\vspace*{15\p@}        \hrule
				\vskip 60\p@
			}}

\begin{document}

				\bigskip
				
				\begin{center}
					
					\textbf{\Large{Goodwillie calculus in the category of algebras over a chain complex  operad}}\\
					
					\bigskip
					
					\title{Equivalences of operads over symmetric monoidal categories}
					\bigskip
					MIRADAIN ATONTSA NGUEMO 
					\bigskip
					
					Universit\'{e} Catholique de Louvain, IRMP, Louvain La Neuve, 1348, Belgium

					\bigskip
					Email: miradain.atontsa@uclouvain.be
					\bigskip

					\bigskip
					\bigskip
				\end{center}
				\begin{Abstract}
					The goal of this paper is to furnish a literature on Goodwillie calculus for functors defined between categories which derive from chain complexes over a ground field $\Bbbk.$
 We characterize homogeneous functors $F: \mathcal{C} \longrightarrow \mathcal{D}$ where $\mathcal{C} ,\mathcal{D}= Ch$ (chain complexes), $Ch_+$(non-negatively graded chain complexes) or $\text{Alg}_\mathcal{O}$ (algebras over a chain complex operad $\mathcal{O}$). In the particular case when 		
	$\mathcal{D}= \text{Alg}_\mathcal{O},$ our characterization requires $\Bbbk$ to be of characteristics $0.$						
	
					We are then extending the results of Walter \cite{Walt06}  who studied in characteristics $0$ the chain complex cases and when $\mathcal{O}$ is the Lie operad.
				\end{Abstract}
				AMS Classification numbers.  Primary: 18D50, 55P65 ;  Secondary:  18G55,   55U35, 55U15
				\begin{Key words}
					Operads, algebras over an operad, model category, calculus of functors
				\end{Key words}
				
				\section*{Introduction}
	The chain complexes are over a ground field $\Bbbk.$ Let $\mathcal{O}$ be a fixed operad on $Ch_+.$
		In this paper, we give a characterization of homogeneous functors $F: \mathcal{C} \longrightarrow \mathcal{D},$ where $\mathcal{C}$ and $\mathcal{D}$ are either $\text{Alg}_\mathcal{O}, Ch_+$ or $Ch.$ This is an "algebraic" version of a couple of publications in Functor Calculus. It starts with  Goodwillie \cite{GIII03} when $\mathcal{C}$ and $\mathcal{D}$ are the category of pointed topological spaces or the associated category of spectra (S-modules of \cite{EKMM}). He proved that homogeneous functors $F$ are completely described in term of symmetric sequences (of spectra) $\partial_*F,$ called "Derivatives". The Functor Calculus for continuous functors was extended by Kuhn, in \cite{Kuhn07}, to the case the categories $\mathcal{C}$ and $\mathcal{D}$ respect the following conditions:
		\begin{enumerate}
			\item  The categories $\mathcal{C}$ and $\mathcal{D}$ are simplicial or topological pointed model categories;
			\item $\mathcal{C}$ and $\mathcal{D}$ are proper: the pushout of a weak equivalence by a cofibration is a weak equivalence, and dually for the pullbacks.
			\item filtered homotopy colimit in $\mathcal{D}$  commutes with finite homotopy limits.
		\end{enumerate}
		By inspection, one see that under good conditions on the functors, the Kuhn (or Goodwillie) Taylor tower construction can be defined  when only the conditions $2.$ and $3.$ are satisfied. 	In our cases the categories $\text{Alg}_\mathcal{O}, Ch_+$ and $Ch$ satisfy the Kuhn's requirements $2.$  and $3.$ We can therefore follow his lines to develop the approximation of  functors. On the other hand, to get a characterization theorem for homogeneous functors similar to Goodwillie's result, Kuhn needed condition 1. 
	In our case  even if the categories $Ch,$ $Ch_+$ and   $\text{Alg}_\mathcal{O}$ (by Hinich in  \cite[$§$ 4.8]{Hinich}) are simplicial categories,  there is not a genuine tensoring of these categories  over sSet. 	Hence, these are not simplicial model categories.
At this point, our  constructions will deviate a little from the literature and we will replace continuous functors by homotopy functors which is a weaker requirement.
		

			\subsection*{Main results} 
		We develop tools in the category $\text{Alg}_\mathcal{O}$ to analyze the Taylor tower of homotopy functors.  In fact we give an explicit model of the homotopy pullbacks and we deduce that any loop  $O$-algebra has a trivial $\mathcal{O}$-algebra structure when the ground field is in characteristics $0.$ On the other hand, we give an explicit model of homotopy pushouts. As a consequence, we deduce that the suspension of an $\mathcal{O}$-algebra is equivalent to a free $\mathcal{O}$-algebra. We apply these constructions to show that a homogeneous functor $F$ is completely described by a  symmetric sequences of unbounded chain complexes  denoted $\partial_*F.$ In fact we analyze the Taylor tower $P_nF$ of $F$ and  we show that under good conditions, there is a weak equivalence 	\begin{center}
								$D_nF(X)\simeq \Omega^{\infty} (\partial_*F\underset{h\Sigma_n}{\otimes} (\Sigma^{\infty}X)^{\otimes n}).$
							\end{center}
		where $D_nF=\text{hofib }(P_nF\longrightarrow P_{n-1}F).$ In our construction,  the pair $(\Sigma^{\infty}, \Omega^{\infty})$ has a different concept from the Kuhn's construction. Namely, in Kuhn's paper  \cite{Kuhn07}, given a simplicial category $\mathcal{D},$ the pair 
			\begin{center}
				$\Sigma^{\infty}: \mathcal{D} \rightleftarrows Spectra(\mathcal{D}): \Omega^{\infty}$
			\end{center}  is an adjoint pair. 
			As in $\cite{Basterra2005},$ when $\mathcal{D}=\text{Alg}_\mathcal{O},$  we will identify  the category $Spectra(\mathcal{D})$ with $Ch.$ We know that these two categories are related by a zig-zag of Quillen equivalences. This identification will  imply a non canonical modification in the construction of these functors. Roughly speaking,  $\Sigma^{\infty}: \text{Alg}_\mathcal{O} \longrightarrow Ch$ becomes the Quillen homology $TQ(-),$ and $\Omega^{\infty}: Ch \longrightarrow \text{Alg}_\mathcal{O} $ assigns to each chain complex  a trivial $\mathcal{O}$-algebra structure. The pair $(\Sigma^{\infty}, \Omega^{\infty})$ we get is no more an adjoint pair. We give more detail about these constructions in section \ref{Algebra_over_operad}. 
			
	When $\mathcal{C}$ and $\mathcal{D}$ are either $Ch$ or $Ch_+,$ or when $\mathcal{O}=Lie$  and the ground field is in characteristics 0, these constructions and results appear in \cite{Walt06}.

				\subsection*{Outline of the paper}
		In the first section \ref{Section Preliminaries}, we  briefly remind the preliminaries on the categories $\text{Alg}_\mathcal{O},Ch_+$ and $ Ch.$   
		 In section  \ref{Section Homotopy limits and homotopy colimits} we make an explicit construction of homotopy pullbacks and homotopy pushouts in $\text{Alg}_\mathcal{O}.$
			We remind	in section \ref{Section Goodwillie Calculus}  the Goodwillie approach in Functor calculus.
				 In section \ref{Section Characterisation of homogeneous functors}, we characterize homogeneous functors. Our method in that section is inspired by the Goodwillie's approach. Namely we  prove that homogeneous functors are infinite loop spaces (as in \cite[Thm 2.1]{GIII03}). We use the "stabilization" of the cross effect and at the end of this process, we obtain a similar (with Goodwillie) characterization of homogeneous functors using the Quillen Homology TQ(-) viewed as $\Sigma^{\infty}.$ 
				We compute	in section \ref{Section Example computing Goodwillie derivatives}  the derivatives of some functors. 
					In the last section \ref{Section Chain rule on the composition} , we prove that there is a chain rule property associated to the composition of two functors (composed on $Ch$).
				
			\subsection*{Future Work}	
				The work of this paper raises the question of extending the description of homogeneous functors to a classification of Taylor towers. This question was raised by Arone-Ching \cite{AC11} and they investigated the module structure on the collection of derivatives $\partial_*F$ over a certain operad. They built a tower which fails to be the Taylor tower of $F$ up to Tate cohomology. Since  Tate cohomologies vanish rationally ( see \cite{Kuhn2004}), this question is studied in our future work (\cite{MiradainGoodwillie2}) when the ground field $\Bbbk$ is of characteristics $0.$ 
				\begin{Acknowledgements}
					The  author is grateful to  Pascal Lambrechts and Greg Arone for their suggestions and  encouragement during this work. 		
				\end{Acknowledgements}

				
				\section{preliminaries}\label{Section Preliminaries}
	
\subsection{Background on chain complexes}
All chain complexes are over a field $\Bbbk $ of  any characteristics.
The purpose of this section is to fix conventions and review basic properties which are background of our constructions.

In this paper, we denote by $Ch$ \index{$Ch$} the category of  $\mathbb{Z}$- graded  chain complexes over $\Bbbk.$ 
This category has a symmetric monoidal structure. The tensor product of chain complexes $V,W\in Ch$ is defined by:
\begin{center}
	$(V\otimes W)_n:= \underset{p+q=n}{\oplus} V_p \otimes W_q$
\end{center}
with the  differential such that: $\forall x\otimes y \in V_p \otimes W_q, $ $d(x\otimes y)= d(x)\otimes y + (-1)^p x\otimes d(y).$
The unit of the monoid $-\otimes-,$ which we  denote abusively  $\Bbbk,$ is the chain complex having $\Bbbk$ in degree $0$ and is trivial in all other degrees.
The  tensor product $-\otimes-$ has a right adjoint $\underline{hom}(-,-)$ given by:
\begin{center}
	$	\underline{hom}(V,W):= \underset{i\in \mathbb{Z}}{\oplus} \underline{hom}^{i}(V,W)$\index{ $ \underline{hom}(-,-)$}
\end{center}
where $\underline{hom}^{i}(V,W)$ denotes the vector space of morphisms $f: V_* \longrightarrow V_{*+i}$ of degree $i.$

Similarly We denote by $Ch_{+} ,$\index{$ Ch_{+}$} the sub-category of  $Ch $ which consist of non negatively graded  chain complexes. 

\subsubsection*{Twisted chain complex}
Let $(V,d_V)$ be a chain complex. A \emph{twisting homomorphism} of degree $-1,$ $d:V\longrightarrow V $ is a morphism of graded vector spaces of degree $-1$ which is added to the internal differential $d_V$ to produce a new differential $d_V+d: V\longrightarrow V$ on $V.$ The equation $(d_V+d)^2=0$ is equivalent to the equation $d_V(d)+ d^2=0,$
with $d_V(d):=d_Vd+dd_V.$ 

\subsubsection*{Model category structure on $Ch_+$}
The category $Ch_+$ is a proper closed model category (for instance see \cite[$§$ 4.2]{GJ94}): 	weak equivalences are quasi-isomorphisms, 
 fibrations are surjections and cofibrations are injections.
In this model category, all objects are cofibrant and fibrant.

\subsection{Operads}
We denote by $FinSet$ the category whose objects are finite sets and whose morphisms are bijections.
We denote the category of all symmetric sequences in $Ch_+$ by $[FinSet, Ch_+]$ (in which morphisms are natural transformations).
 The composition $M\circ N,$ of the two symmetric sequences $M$ and $N,$ is defined by:
\begin{center}
	$(M\circ N)(J):= \underset{J=\underset{j\in J'}{\amalg}J_j}{\bigoplus} M(J') \otimes \underset{j\in J'}{\bigotimes}N(J_j).$
\end{center}
The coproduct here is taken over all unordered partitions, $\{J_j\}_{j\in J'},$  of $J.$
 The unit symmetric sequence $\mathbb{I}$ is given by 
\begin{center}
	$\mathbb{I}(J)=\Bbbk,$ if $|J|=1,$ and $\mathbb{I}(J)=0$ otherwise; 
\end{center}
\begin{definition}[Operads]
	An operad in $Ch_+$ is a monoid  over $([Finset, Ch_+], \circ, \mathbb{I}).$ A reduced operad is an operad $\mathcal{O}$ such that $\mathcal{O}(0)=0$ and $\mathcal{O}(1)=\Bbbk.$ In this paper, we only consider reduced operads in $Ch_+.$
\end{definition}

\subsection{Algebra over an operad}\label{Algebra_over_operad}
Let $\mathcal{O}$ be a reduced operad. 
An $\mathcal{O}-$algebra $X$ is a chain complex together with structure maps, for any $n\geq 0$: 
$ m_n : \mathcal{O}(n)\underset{\Sigma_n}{\otimes} X^{\otimes n}\longrightarrow X,$ 
satisfying the appropriate compatibility conditions.
Maps of $\mathcal{O}-$algebras are given by chain complex morphisms  $f:X\longrightarrow X'$ which are degree $0$ and preserve the $\mathcal{O}-$algebra structures of $X$ and $X'.$ The category of $\mathcal{O}$-algebras is denoted $\text{Alg}_\mathcal{O}.$ \index{$\text{Alg}_\mathcal{O}$}

\subsubsection*{Model category structure on $\text{Alg}_\mathcal{O}$}

One use this adjunction $\mathcal{O} (-): Ch_{+}\rightleftarrows \text{Alg}_\mathcal{O}: U$ , between the forgetful and the free functors, to define the projective model structure on $ \text{Alg}_\mathcal{O}$ (see \cite[Thm 4.4]{GJ94}).
Namely weak equivalences(resp. fibrations) of $\text{Alg}_\mathcal{O}$  are  equivalences (resp. fibrations) in the underlined category $ Ch_{+}.$ The cofibrations are morphisms having the right lifting property with respect to acyclic fibrations. In particular, cofibrant $\mathcal{O}$-algebras are retract of quasi-free algebras.

\subsection{Cooperad}\label{Cooperad}
The notion of cooperad is dual to the notion of operad. The dual notion consists of considering the opposite category $((Ch_{+})^{op}, \otimes, \mathbb{I}_{Ch}).$ We define the dual composition product $\widehat{\circ}$ of two symmetric sequences by replacing the coproduct in the definition \ref{Symmetric sequence} with a product. That is 
\begin{center}
	$ ( M\widehat{\circ}N)(J):=\underset{J=\underset{j \in J'}{\amalg} J_j}{\prod}M(J')\otimes \underset{j\in J'}{ \bigotimes}N(J_j).$
\end{center}
where the product is taken over all unordered partitions, $\{J_j\}_{j\in J'},$ of $J.$ 

A cooperad in $\mathcal{C}$ is a triple $(Q, m^c, \eta^c),$ where $Q$ is a symmetric sequence together with maps \begin{center}
	$m^c: Q\longrightarrow Q\widehat{\circ} Q$ and $\eta^c: Q\longrightarrow \mathbb{I}$
\end{center}
satisfying the co-associativity, the left and right co-unit condition.

A cooperad $Q$ is connected when $\widetilde{Q}:= ker(\eta^c)$ is concentrated strictly in positive degree. 

Since (finite) product and direct sum are equivalent in the underlying category $Ch_+,$ in the rest of this thesis, the dual composition product $\widehat{\circ}$ will simply be denoted $\circ.$

\subsection{Coalgebra over a cooperad}\label{Section Coalgebra over a cooperad}

Another dual analogy is the notion of the coalgebra over a cooperad. That is, any chain complex $Y$ together with a structure map, $\forall n,$ $m_n^{c}: Y\longrightarrow Q(n)\underset{\Sigma_n}{\otimes} Y^{n}$ satisfying the appropriate compatibility conditions. The maps of $Q$-coalgebras are degree 0 chain complex morphisms $f:Y\longrightarrow Y'$ which preserves the structures of $Y$ and $Y'.$  One denotes the category of $Q$-coalgebras by $\text{coAlg}_Q.$ 
\subsubsection*{Model category on $\text{coAlg}_Q$ }
We use this adjunction 	$U: \text{coAlg}_Q \rightleftarrows Ch_{+}: Q (-)$ between the, forgetful and the cofree functor, to define an injective model structure on $\text{coAlg}_Q$ (see \cite[Thm 4.7]{GJ94}). Namely weak equivalences(resp. cofibrations) of $\text{coAlg}_Q$ are weak equivalences(resp. cofibrations) in the underlined category weak  $Ch_{+}.$ The fibrations are morphisms having the left lifting property with respect to acyclic cofibrations

\subsection{Bar and Cobar constructions}

\subsubsection{Bar construction}

\subsubsection*{$J$-tree}	
Let $J$ be a finite set. A $J-$tree is an abstract planar tree with one output edge on the bottom, and input edges on the top whose sources also called \emph{leaves} are indexed by $J.$ These input edges and the edge from the root are the \emph{ external edges } of the tree, and the other edges are called \emph{internal} edges. The vertices of internal edges are called \emph{internal vertices}. Given an $J-$tree $T,$ we denote by $V(T)$ the set of its internal vertices, and  $E(T)$ the set of edges. 
The set of $J-$trees, denoted by $ \beta(J),$ is equipped with a natural groupoid structure. Formally, an isomorphism of $J-$trees $ \beta: T'\longrightarrow T$ is defined by bijections $\beta_V: V(T')\longrightarrow V(T)$ and $\beta_E: E(T')\longrightarrow E(T)$ preserving the source and target of edges. In other word, $ \beta(J)$ is the groupoid of $J-$labeled trees and non-planar isomorphisms.

\begin{definition}[Free object]\label{Free object}
	Let $M$ be a symmetric chain complex. The \emph{free object}, associated to $M,$  and denote by $F(M)$ consists of: chain complexes  $(F(M)(J), \partial_0),$ for any finite set $J,$  defined as
	\begin{center}
		$F(M)(J)=\underset{T \in \beta(J)}{\oplus}T (M)/\equiv$
	\end{center}
	where $T (M)=\underset{v\in V(T)}{\otimes}M(J_v),$ and the equivalence classes are made of non planar isomorphisms of $J-$trees. The differential $\partial_0$ is induced naturally by the differentials of the chain complexes $(M(J_v), \partial_{J_v}).$
	
	A bijection $ J\longrightarrow J'$ gives an isomorphism $F(M)(J)\longrightarrow F(M)(J')$ by relabeling the leaves of the underlined trees. In this way $F(M)$ becomes a symmetric sequence in chain complexes.
	
\end{definition}

Let $\mathcal{O}$ be an operad, $R$  is a right $\mathcal{O}$-module and $L$ is a left $\mathcal{O}$-module.

\begin{definition} [Two sided bar construction]\label{Two side bar construction}
	The two sided bar construction $B(R, \mathcal{O}, L)$ is the symmetric sequence of chain complexes given by: for any finite set $J,$
	\begin{center}
		$B(R, \mathcal{O}, L)(J):=( R\circ  F(s\widetilde{\mathcal{O}})   \circ L(J), \partial_0+\partial),$ with $\widetilde{\mathcal{O}}= ker\varepsilon.$\index{$B(R, \mathcal{O}, L)$}
	\end{center}
	The differential $\partial_0$ is induced in the natural way by the differentials of the chain complexes $\{ ( R(J'), d_{J'}) \}_{J'\subseteq J},$ $\{ ( \mathcal{O}(J'), d_{J'}) \}_{J'\subseteq J},$ and $\{ ( L(J'), d_{J'}) \}_{J'\subseteq J}.$ The second differential $\partial=\partial_R+\partial_\mathcal{O}+\partial_L$ of this complex is the derivation which integrates the structure morphisms:
	$m_R: R\circ \mathcal{O} \longrightarrow R, m_L:\mathcal{O}\circ L \longrightarrow L,$ and $m_\mathcal{O}:\mathcal{O}\circ \mathcal{O} \longrightarrow \mathcal{O}$ ( for explicit description, see   \cite[$§$ 4.4.3.]{BF04}).
	
\end{definition}

If $L=R= \mathbb{I}, $ then $B(R, \mathcal{O}, L)$ is the usual bar construction $B(\mathcal{O}).$

\subsubsection*{Bar construction with coefficients in $\mathcal{O}$-algebras}	

Given an $\mathcal{O}$-algebra $X,$ there is an associated left $\mathcal{O}$-module $\hat{X}$ defined as follows:
\begin{center}
	$ \hat{X}(\underline{0})=X, \hat{X}(\underline{n})=0$ if $n\geq 1$ 
\end{center}\index{$\widehat{X}=(X, 0, 0, ...)$ is the left $\mathcal{O}$-module associated to $X$}
and the left action $\mathcal{O}\circ \hat{X} \longrightarrow \hat{X}$ is induced in the obvious way by the $\mathcal{O}$-algebra structure of $X.$

This defines an embedding functor $\widehat{-}: \text{Alg}_\mathcal{O} \longrightarrow \mathcal{O}$-mod
of $\mathcal{O}$-algebras to the category of left modules over  $\mathcal{O}.$ We use this embedding  and the bar construction with coefficients in left and right $\mathcal{O}$-modules described in Definition \ref{Two side bar construction} to define the bar construction with coefficients in $\mathcal{O}$-algebras.

\begin{definition}[Bar construction on algebras]	
	Let $X$ be an algebra over a reduced operad $\mathcal{O}.$ We define the bar construction  on $\mathcal{O}$ with coefficient in $X$ as the chain complex: 
	\begin{center}
		$B(\mathcal{O}, X):=\underset{n}{\oplus}(B(\mathbb{I}, \mathcal{O}, \hat{X})(\underline{n}), \partial_0+\partial ),$
	\end{center} 
\end{definition}

\subsubsection{Cobar construction} 
Let $(Q, Q \overset{m^c}{ \longrightarrow }Q \circ Q,  Q \overset{\eta^c}{ \longrightarrow }\mathbb{I})$  be a connected cooperad in $Ch_+,$ and denote $\widetilde{Q}:= ker(\eta^c).$
The cobar construction of $Q,$ denoted $B^c(Q)$ is the dual version of the bar construction (for operads). Namely, this is the  quasi-free operad
\begin{center}
	$B^c(Q)=(F(s^{-1}\widetilde{Q}), \partial_0 + \partial^*),$
\end{center}
where $\partial_0$ is the internal differential of $F(s^{-1}\widetilde{Q})$ induced by that of $Q,$ and $\partial^*$ is the differential defined by reversing all the arrows in the definition of $\partial$ on $B(\mathcal{O})$ in Definition \ref{Two side bar construction} (when $L=R= \mathbb{I}$).

In this definition, the cooperad $Q$ needs to be connected to avoid the case where $B^c(Q)$ has elements in negative degree. We are now ready to state the next theorem (in characteristics 0) which gives a duality between  the bar construction and the cobar construction.

\begin{theorem}\cite[Theorem 2.17]{GJ94}
	The functors $B^c$ and $B$ form an adjoint pair 
	between the categories of connected cooperads and augmented operads.	
\end{theorem}
In addition, it is proved in \cite[Theorem 3.2.16]{GK95} that the unit $Q \longrightarrow BB^c(Q)$ and the counit $B^cB(\mathcal{O}) \longrightarrow \mathcal{O}$ of this adjunction are quasi-isomorphisms.

\subsubsection*{Cobar construction with coefficient in $Q$-coalgebras}
Let $Q$ be a connected cooperad on chain complexes and $Y$ a $Q$-coalgebra.
One can follow the same procedure such as in the case of bar construction over algebras to define the cobar construction $B^{c}(\mathcal{O}, Y).$ In this sense we will get literally an object in the category of $B^{c}(Q)$-algebra. We do not follow these steps here
since we will need further as an application a cobar construction functor which send a $B(\mathcal{O})$-coalgebra (for a given reduced operad $\mathcal{O}$) into an $\mathcal{O}$-algebra.

We consider to have from now a reduced operad  $\mathcal{O}$  such that $B^{c}(Q)\overset{\simeq}{\longrightarrow} \mathcal{O}.$
This later morphism induces a degree $0$ morphism  $s^{-1}\widetilde{Q}\longrightarrow \mathcal{O}$ which gives a morphism $\theta: \widetilde{Q}\longrightarrow \mathcal{O}$ of degree $-1.$ We use $\theta$ to define the composition

\begin{center}
	$\xymatrix{
		w: Q(Y) \ar[rr]^-{ Q (m^{c}_Y) }&& Q (Q (Y))\ar[rr]^-{ \theta  (1_{Q Y}) }   & &\mathcal{O} (Q (Y)) 
	}$        
\end{center}
The derivation $d_w: \mathcal{O} (Q (Y)) \longrightarrow \mathcal{O} (Q (Y))$ of degree $-1$ associated to this $w$ satisfies the equation of twisting homomorphism $d(w)+ d_w .w=0$ on $Q(Y).$
This is equivalent to say that $(\mathcal{O} (Q (Y)), d+d_w )$ is a quasi-free $\mathcal{O}$-algebra.  The morphism $\theta$ which is at the base of this construction will be called in the literature \emph{twisting cochain} (see \cite[def 2.16]{GJ94}).
\begin{definition}[cobar construction on a $Q$-coalgebra] Let $Y$ be a $Q$-coalgebra.  
	The cobar construction on $Y,$ associated to the twisting cochain $\theta: \widetilde{Q}\longrightarrow \mathcal{O}$, and denoted $B^{c}_{\theta}(Q, Y)$ is the quasi-free $\mathcal{O}$-algebra 
	\begin{center}
		$B^{c}_{\theta}(Q, Y)= (\mathcal{O} (Q (Y)), d+d_w )$
	\end{center}
	where $d$ is the internal differential of $\mathcal{O} (Q (Y))$ induced by the complexes $\mathcal{O}, Q$ and $Y.$
	
	When $Q=B(\mathcal{O}),$ the map $\theta$ that we consider is given by the projection $B(\mathcal{O})\longrightarrow s\widetilde{\mathcal{O}}.$ In that specific case, we will always drop $\theta$ and simply write $B^{c}(Q, Y)$ to mean

	as $\theta$ is given by the projection $B(\mathcal{O})\longrightarrow s\widetilde{\mathcal{O}}.$
\end{definition}

One form the cobar-bar  adjoint pair  
\begin{center}	
	$B^{c}(B(\mathcal{O}),-):     \text{coAlg}_{B(\mathcal{O})}   \rightleftarrows 
	\text{Alg}_\mathcal{O}:	B(\mathcal{O},-)$
\end{center}

whose the unit and co-unit functors have the following property(in characteristics 0):
\begin{theorem}[\cite{GJ94}, Theorem 2.19]\label{GJ} Given an $\mathcal{O}$-algebra $X$ and a $B(\mathcal{O})$-coalgebra $Y,$  
	the co-unit $B^{c}(B(\mathcal{O}),B(\mathcal{O},X))\longrightarrow X$ and the unit $Y\longrightarrow B(\mathcal{O},B^{c}(B(\mathcal{O}),Y))$ are weak equivalences.
\end{theorem}
With the model structure defined on $ \text{coAlg}_{B(\mathcal{O})}$ and $ \text{Alg}_\mathcal{O},$ we can see that the cobar-bar adjunction is actually a Quillen pair, and Theorem \ref{GJ} completes in proving that this adjunction is a Quillen equivalence.

\begin{remark}\label{Remark cobar-bar duality in any characteristics}
	Fresse proved in \cite[Prop 3.1.12]{BF04} that  the counit $B^cB(\mathcal{O}) \longrightarrow \mathcal{O}$ is a quasi-isomorphism when the ground field $\Bbbk$ is of any characteristics. Thus one can apply this fact in \cite[Thm 4.2.4]{BF09} to deduce that the co-unit $B^{c}(B(\mathcal{O}),B(\mathcal{O},X))\longrightarrow X$ is a cofibrant resolution of $X$ if $\Bbbk$ is a field of any characteristics.
\end{remark}

\subsection{Definition of the functors $\Omega^{\infty}$ and $\Sigma^{\infty}$}

Let $X$ be an $\mathcal{O}$-algebra.  We define the functor $\tau_1\mathcal{O}\underset{\mathcal{O}}{\circ} -: 	\text{Alg}_\mathcal{O} \longrightarrow 	\text{Alg}_\mathcal{O}$ as follows:\begin{center}
	$\tau_1\mathcal{O}\underset{\mathcal{O}}{\circ} X:=colim_{\text{Alg}_\mathcal{O} }(\tau_1\mathcal{O}\circ \mathcal{O}\circ (X)\rightrightarrows \tau_1\mathcal{O}\circ (X)),$
\end{center}
where $\tau_1\mathcal{O}$ is the operad: $\tau_1\mathcal{O}(1)=\Bbbk,$ and $\tau_1\mathcal{O}(t)=0$ if $t\neq 1;$ The first map of this colimit is produced by the multiplication $\tau_1\mathcal{O}\circ \mathcal{O} \longrightarrow \tau_1\mathcal{O}$ and the second map is given by the algebra structure map $\mathcal{O} (X)\longrightarrow X.$ Strictly speaking, the algebra $\tau_1\mathcal{O}\underset{\mathcal{O}}{\circ} X$ has a trivial $\mathcal{O}$-algebra structure, thus we will define the abelianization functor as its composite with the forgetful functor
\begin{center}
	$(-)^{ab}: 	\text{Alg}_\mathcal{O} \overset{ \tau_1\mathcal{O}\underset{\mathcal{O}}{\circ}-}{\longrightarrow} \text{Alg}_\mathcal{O} \overset{U}{\longrightarrow }Ch_{+}\overset{I}{\hookrightarrow} Ch$
\end{center}
where $I$ is the inclusion fucntor defined by $I(V)_t:=\left\{  \begin{array}{ll}
V_t & \mbox{if  }  t\geq 0\\
0 & \mbox{if }  t<0
\end{array} \right.$ 

The abelianization functor has a right Quillen adjoint functor:
\begin{center}
	$\Omega^{\infty}: Ch\overset{red_0}{\longrightarrow} Ch_{+} \overset{(-)_{triv}}{\longrightarrow } \text{Alg}_\mathcal{O}$\index{$\Omega^{\infty}$}
\end{center}
where for any chain complex $C_*,$ $red_0(C_*)_t:=\left\{  \begin{array}{ll}
C_t & \mbox{if  }  t>0\\
ker(d_0) & \mbox{if }  t=0
\end{array} \right.$ 

and $(-)_{triv}$ is the functor which assigns to any non negative chain complex the trivial $\mathcal{O}$-algebra structure. 

The functor $(-)^{ab}$ does not preserves quasi-isomorphisms in general, apart from preserving quasi-isomorphisms between quasi-free algebras (since they are cofibrant objects in $\text{Alg}_\mathcal{O}$). Its  associated derive functor is called in the literature Quillen homology.

\begin{definition}[Quillen homology]
	If $X$ is an $\mathcal{O}$-algebra, then the \emph{Quillen homology} $TQ(X)$ of $X$ is the $\mathcal{O}$-algebra $\tau_1\mathcal{O}\overset{h}{\underset{\mathcal{O}}{\circ}} X.$
\end{definition}			
Again since the algebra structure on $TQ(X)$ is trivial, we will abuse notation and consider it as an object in $Ch_+.$ 
We will give in the next lines an explicit model of the functor $TQ(-)$\index{$TQ(-)$} which we  will need to define $\Sigma^{\infty}.$

Let $X$ be an algebra over $\mathcal{O}.$ One associate to $X$ the symmetric sequence $\hat{X}$ defined as $\hat{X}(k)=X,$ if $k=0$ and $\hat{X}(k)=0,$ if $k\neq 0. $ One can see that $\hat{X}$ is a left $\mathcal{O}$-module with the structure map induced naturally by the algebra structure map. In this sense the category $\text{Alg}_\mathcal{O}$ embeds in $\text{Lt}_\mathcal{O}$( the category of left $\mathcal{O}$-modules) as a full subcategory of left modules concentrated at degree $0,$ via the functor $\hat{-}: \text{Alg}_\mathcal{O}\longrightarrow \text{Lt}_\mathcal{O}, X\longmapsto \hat{X}.$

According to Theorem \ref{GJ}, $B^c(B(\mathcal{O}), B(\mathcal{O}, X))$ is a cofibrant replacement of $X,$ therefore $TQ(X)\simeq UB(\mathcal{O}, X),$ where  $U: \text{Coalg}_{B(\mathcal{O})}\longrightarrow Ch_+$ is the forgetful functor.
Under this last quasi-isomorphism, we will consider the functor $UB(\mathcal{O}, -)$ as our explicit model for the functor $TQ(-)$ and we will denote by $\Sigma^{\infty}$ \index{$\Sigma^{\infty}$} the composition: 
\begin{center}
	$	
	\xymatrix{\text{Alg}_{\mathcal{O}} \ar@/_1pc/[rrrrrr]|{\Sigma^{\infty} }   \ar[rr]^{B(\mathcal{O}, -)}  & & \text{Coalg}_{B(\mathcal{O})} \ar[rr]^-{U}&  &Ch_{+} \ar@{^{(}->}[rr]^{I}& & Ch}
	$
\end{center}	
Before concluding these constructions, we make the following remark:

We consider the following two adjunctions
\begin{center}
	$
	\xymatrix{
		\text{coAlg}_{B(\mathcal{O})} \ar@<1ex>[rr]^-{U}  & &Ch_+  \ar@<1ex>[rr]^-{I}  \ar@<.5ex>[ll]^-{  B(\mathcal{O}, (-)_{triv})}& &\ar@<.5ex>[ll]^-{  red_0} Ch}, $
\end{center}
where the top functors are each left adjoint  and the bottom functors are each right adjoint. We then observe that the associate comonad is $I U B(\mathcal{O}, (-)_{triv})red_0 \cong \Sigma^{\infty}\Omega^{\infty}.$ Therefore even if 
the functor $ \Omega^{\infty}$ is not adjoint to $\Sigma^{\infty},$ we can say that  $T=\Sigma^{\infty}\Omega^{\infty}: Ch \longrightarrow Ch$\index{$\Sigma^{\infty}\Omega^{\infty}$} is a comonad and  is a "homotopy good model" of the comonad $(-)^{ab}(-)_{triv}.$

We extend construction of the functors $\Sigma^{\infty}$ and $\Omega^{\infty}$ to other categories as follows:
	\begin{enumerate}
		\item[-] $\Sigma^{\infty}:=I: Ch_+ \longrightarrow Ch;$
		\item[-] $\Sigma^{\infty}=Id: Ch \longrightarrow Ch;$ 
		\item[-] $\Omega^{\infty}=red_0 -: Ch \longrightarrow Ch_+; $
		\item[-] $\Omega^{\infty}=Id: Ch \longrightarrow Ch.$
	\end{enumerate}

\section{Homotopy limits and colimits in $\text{Alg}_\mathcal{O}$ }\label{Section Homotopy limits and homotopy colimits}
 The purpose of this section is to remind a brief notion of homotopy limits and colimits, and give their explicit description in  $\text{Alg}_\mathcal{O}$  in terms of holims and hocolims in chain complexes.

Let $\mathcal{C}$ and $\mathcal{D}$ be any of the categories $\text{Alg}_\mathcal{O},\text{coAlg}_{B(\mathcal{O})}$ and  $Ch_{+}.$ 
These categories are complete and cocomplete. The authors of \cite{DHKS} proved, in a general argument for complete and cocomplete model categories,  that holims and hocolims always exists in $\mathcal{C}$(see \cite[19.2]{DHKS}). More explicitly, given a small category $J,$ and an $J$-diagram $D$ in $\mathcal{C},$ they replace $D$ through a functor $D\longmapsto D_{vf}$ (resp. $D\longmapsto D_{vc}$) which associate a so called "virtually fibrant replacement" (resp "virtually cofibrant replacement") $D_{vf}$ (resp. $D_{vc}$) such that there is a map $D\overset{\simeq}{\longrightarrow} D_{vf}$ (resp. $D_{vc}\overset{\simeq}{\longrightarrow} D$) natural in $D.$ 

According to this vocabulary  we can now set the definition of holims and hocolims:
\begin{definition} Given an $J$-diagram $D$ in $\mathcal{C},$\begin{center}
		$holim_{\mathcal{C}}(D):=lim_{\mathcal{C}}(D)_{vf}$ and 	$hocolim_{\mathcal{C}}(D):=colim_{\mathcal{C}}(D)_{vc}.$
	\end{center}
\end{definition}

\subsection{Homotopy pullback in $\text{Alg}_\mathcal{O}$ }
We assume in this section that the ground field $\Bbbk$ is of characteristics 0.
The homotopy limit in $\text{Alg}_\mathcal{O}$ is calculated using observations in the underlined category $Ch_{+}.$ Given a diagram $D: X \overset{g}{\longrightarrow} Z \overset{f}{\longleftarrow} Y $ in $Ch_{+},$ if either $f$ or $g$ is a surjection (to mean fibration), then $holim_{ Ch_{+}}(D)\simeq lim_{Ch_{+}}(D).$  This comes out easily when we apply the homology long exact sequence theorem to the two parallel fibrations of the pullback square associated to $D.$  Our methodology to define an explicit homotopy limit in $\text{Alg}_\mathcal{O}$ is to replace the maps of the $\mathcal{O}$-algebra diagram by explicit surjections.
\subsection*{Construction of path objects in $\text{Alg}_\mathcal{O}$ }
Let $\mathcal{I}=(\wedge(t,dt),d)$ be the free differential graded commutative algebra generated by the element $t$ in degree $0$ and $dt$ in degree $-1, $ with differential $d$ given by $d(t)=dt$ and $d(dt)=0.$ It is useful to notice that an element $\alpha$ of $\mathcal{I}$ has the form $\alpha=P(t)+Q(t)dt$ with $P,Q\in \Bbbk [t].$

There are natural commutative algebra maps $s_0: \Bbbk\longrightarrow \mathcal{I}$ and $p_0, p_1: \mathcal{I}\longrightarrow \Bbbk$ defined as: $\forall (\alpha=P(t)+Q(t)dt\in \mathcal{I})$ and $k\in \Bbbk,$
\begin{center}
	$p_0(\alpha):=P(0) ,$   $ p_1(\alpha):= P(1)$ and $s_0(k)=k$
\end{center}
$s_0$ is a quasi-isomorphism and $p_0s_0=p_1s_0=1_\Bbbk.$

For any $\mathcal{O}$-algebra $X,$ there is a natural $\mathcal{O}$-algebra structure on $\mathcal{I}\otimes X$ (see \cite[$§$2.4]{Livernet}) given by :	If $a\in \mathcal{O}(n), \alpha_i \otimes x_i \in \mathcal{I}\otimes X,$ for $1\leq i\leq n,$
\begin{center}
	$m(a\otimes \alpha_1 \otimes x_1 \otimes ...\otimes \alpha_n \otimes x_n):=\pm  \alpha_1...\alpha_n \otimes m_X(a\otimes x_1\otimes ...\otimes x_n\otimes);$
\end{center}
One then get the factorization in $\mathcal{O}$-algebras (unbounded algebras)
\begin{center}
	$\xymatrix{
		X \ar[r]^-{ s_0\otimes X }&\mathcal{I}\otimes X\ar@<-.5ex>[r]_-{  p_0 \otimes X} \ar@<1ex>[r]^-{p_1 \otimes X}  & X 
	}$        
\end{center}
which yield to the diagram in $\text{Alg}_\mathcal{O}:$

\begin{center}
	$\xymatrix{
		X \ar[r]^-{s_{0}^{X} }&red_0(\mathcal{I} \otimes X)\ar@<-.5ex>[r]_-{  p_{0}^{X}} \ar@<1ex>[r]^-{p_{1}^{X}}  & X 
	}$        
\end{center}
One can prove that $p_{0}^{X}$ and $p_{1}^{X}$ are trivial surjections.
\begin{definition}[path object]
	A path object associated to an $\mathcal{O}$-algebra $X$ is the $\mathcal{O}$-algebra $X^{\mathcal{I}}:=red_0(\mathcal{I} \otimes X)$
	together with the $\mathcal{O}$-algebra morphisms $ p_{0}^{X}, p_{1}^{X}$ and $s_{0}^{X}.$
\end{definition}

\subsection*{Construction of homotopy pullbacks in $\text{Alg}_\mathcal{O}$}
Let us consider the commutative diagram in $\text{Alg}_\mathcal{O}:$
\begin{center}
	\[
	\xymatrix{
		Y \ar@/_/[dr]_f \ar@/^/[drrr]^{1_Y} \ar@{.>}[drr]|-{( s_{0}^{Z}f, Y)}\\
		&Z\ar@/_/[dr]_{s^Z_0} &Z^{\mathcal{I}}\underset{Z}{\times} Y \ar[d]^{\pi_1} \ar[r]_-{\pi_2}^{\simeq} & Y\ar[d]^f \\
		& &Z^{\mathcal{I}} \ar[r]^{p^Z_0}_{\simeq} &Z}
	\]
\end{center}
where the square in the middle is a pullback. From the left triangle, we build the following factorization of $f:$
\begin{center}
	$\xymatrix{
		Y \ar[r]^-{( s_{0}^{Z}f, Y) }_-{\simeq}\ar[r]\ar[rd]_{f=p^Z_1s^Z_0f} & Z^{\mathcal{I}}\underset{Z}{\times} Y\ar[d]^{p_{1}^{Z}\pi_1} \\
		& Z 
	}$        
\end{center} 
We use this factorization to replace $f$ in a diagram $D: X \overset{g}{\longrightarrow} Z \overset{f}{\longleftarrow} Y$ by the fibration $p_1^Z \pi_1.$  

\begin{proposition}\label{holims}
	Given an $\mathcal{O}-$algebra diagram $D: X \overset{g}{\longrightarrow} Z \overset{f}{\longleftarrow} Y, $ a homotopy pullback of $D$ is the  $\mathcal{O}-$algebra $P_{D}= X\underset{Z}{\times} Z^{\mathcal{I}}\underset{Z}{\times} Y,$ namely  $P_{D}=lim_{\text{Alg}_\mathcal{O}}(X \overset{g}{\longrightarrow} Z \overset{p_{1}^{Z}\pi_1}{\longleftarrow} Z^{\mathcal{I}}\underset{Z}{\times} Y ).$	
\end{proposition}

\begin{proof} 
	\begin{enumerate}
		\item The morphism $p_{1}^{Z}\pi_1$ is a surjection,  it then follows from our comment in the introduction of this section that in $Ch_{+},$ we have 
		\begin{align*}
		UP_D\simeq& \text{holim }_{Ch} (UX \overset{g}{\longrightarrow} UZ \overset{p_{1}^{Z}\pi_1}{\longleftarrow} UZ^{\mathcal{I}}\underset{Z}{\times} Y )\\
		\simeq & \text{holim }_{Ch}UD
		\end{align*}
		As a consequence we deduce that the functor $P_{-}$ preserves weak equivalence of diagrams in $\text{Alg}_\mathcal{O}$. Therefore we retain that 
		$P_D\overset{\simeq}{\longrightarrow} P_{D_{vf}}.$
		\item We now prove that $lim_{\text{Alg}_\mathcal{O} }(D_{vf}) \overset{\simeq}{\longrightarrow} P_{D_{vf}}.$ Lets consider  $D_{vf}: X' \overset{g'}{\longrightarrow} Z' \overset{f'}{\longleftarrow} Y'$ and the cube in  $Ch_{+}:$ 
		
		\begin{center}
			\begin{tikzcd}[row sep=scriptsize, column sep=scriptsize]
				&U P_{D_{vf}} \arrow[dl,  two heads, "h_1"] \arrow[rr]  & & UZ'^{I} \underset{Z'}{\times} Y' \arrow[dl,  two heads, "h_2"]  \\ UX' \arrow[rr, crossing over]  & & UZ' \\
				& Ulim_{\text{Alg}_\mathcal{O} }(D_{vf})  \arrow[dl, "h_4"] \arrow[uu,] \arrow[rr] & & UY' \arrow[uu, " (s_{0}^{Z'}f; Y') "] \arrow[dl, "h_3"] \\
				UX' \arrow[uu, "="] \arrow[rr] & & UZ' \arrow[uu, "="] \\
			\end{tikzcd}
		\end{center}
		where we have applied the forgetful functor $U: \text{Alg}_\mathcal{O} \longrightarrow Ch_{+}$ to the original natural cube in $\text{Alg}_\mathcal{O}.$
		The square obtained from the homotopy fibers of the morphisms $h_1, h_2, h_3$ and $h_4$ has the following characteristics:
		\begin{center}
			$\xymatrix{
				hofibre(h_1) \ar[r]^-{(1) }_-{\simeq}\ar[r]&	hofibre(h_2) \\
				hofibre(h_4)\ar[u] \ar[r]_{(3)}^{\simeq}	& 	hofibre(h_3)  \ar[u]^{(2)}_{\simeq}
			}$      
		\end{center}	
		where 
		\begin{enumerate}
			\item [(1)]  is a weak equivalence as the top square of the cube is a homotopy pullback;
			\item[(2)] is a weak equivalence because $1_{UZ'}$ and $U(s_{0}^{Z}f, Y)$ are weak equivalences, and these  imply that the right hand square is a homotopy pullback;
			\item[(3)] is a weak equivalence since the bottom square is a homotopy  pullback.
		\end{enumerate}
		We can conclude that $hofibre(h_4)\longrightarrow hofibre(h_1)$ is a weak equivalence and therefore that $lim_{\text{Alg}_\mathcal{O} }(D_{vf}) \overset{\simeq}{\longrightarrow} P_{D_{vf}}.$
		We then conclude in conclusion that 
		\begin{center}
			$P_D\overset{\simeq}{\longrightarrow} P_{D_{vf}} \overset{\simeq}{\longleftarrow} lim_{\text{Alg}_\mathcal{O} }(D_{vf})$
		\end{center}
	\end{enumerate}
\end{proof}
In general this construction is extended in the obvious manner in order to define higher dimensional limits in $\text{Alg}_\mathcal{O}.$

\begin{mylemma}\label{Trivial_algebra_on_Omega}
	If $X$ is an $\mathcal{O}$-algebra, then the map 
	\begin{align*}
	\Phi: (red_0s^{-1}X)_{triv}&\longrightarrow \Omega X\\
	s^{-1}x&\longmapsto (0,dt \otimes x, 0)
	\end{align*}
	is a weak equivalence in $\text{Alg}_\mathcal{O}.$
\end{mylemma}

\begin{proof} We first prove that $\Phi$ is a map of $\mathcal{O}$-algebras. Namely
	let $x_1, ..., x_n \in X,$ and $a\in \mathcal{O}(n), (n\geq 2),$ then 
	\begin{align*}
	m_{\Omega X}(a\otimes \Phi(s^{-1}x_1)\otimes ...\otimes\Phi(s^{-1}x_n))&=(0,dt^{n}\otimes m_X(a\otimes x_1\otimes ...\otimes x_n), 0)\\
	&=0 \text{  ( since $dt^{n}=0$) }
	\end{align*}
	This computation proves that $\Phi$ is a map of $\mathcal{O}$-algebras as the $\mathcal{O}$-algebra structure on $ (red_0s^{-1}X)_{triv}$ is trivial.
	It is obvious that the map $\Phi$ commutes with differentials of the two complexes.
	
	Now we prove by hand that $H_*(\Phi)$ is injective and surjective. Let us take
	\begin{center}
		$\overline{x}=a_0+\underset{l\geq 1}{\Sigma} t^{l}a_l +\underset{k\geq 0}{\Sigma} t^{k}dt b_k\in X^{\mathcal{I}} $ such that $(0, \overline{x}, 0) \in \Omega X \cap Ker d$ 
	\end{center} where for each $l$ and $k, $   $a_l, b_k \in X;$ 
	\begin{align*}
	(0, \overline{x}, 0)\in \Omega X& \Longleftrightarrow p_1 ^{X}(\overline{x})=0=p_0 ^{X}(\overline{x})\\
	&\Longleftrightarrow a_0=0=\underset{l\geq 1}{\Sigma}a_l
	\end{align*}
	One can also see that 
	\begin{align*}
	d\overline{x}=0 \Longleftrightarrow & \forall l\geq 1,  da_l=0  \text{ and  }
	\underset{l\geq 1}{\Sigma} lt^{l-1}dta_l=\underset{k\geq 0}{\Sigma} t^{k}dt db_k
	\end{align*}
	This last equality implies that $\forall l\geq 1, a_l=\frac{1}{l}db_{l-1}$ and thus $\underset{l\geq 1}{\Sigma}\frac{1}{l}db_{l-1}=0$ 
	One then get:
	\begin{align*}
	\overline{x}&= \underset{l\geq 1}{\Sigma}\frac{1}{l}t^{l}db_{l-1}+\underset{l\geq 1}{\Sigma}t^{l-1}dtb_{l-1}\\
	&=\underset{l\geq 1}{\Sigma}\frac{1}{l}(t^{l}db_{l-1}+ lt^{l-1}dtb_{l-1})\\
	&=d(\underset{l\geq 1}{\Sigma}\frac{1}{l}t^{l}b_{l-1})\\
	&= d(\underset{l\geq 1}{\Sigma}\frac{1}{l}t^{l}b_{l-1}-t\underset{l\geq 1}{\Sigma}\frac{1}{l}b_{l-1}+t\underset{l\geq 1}{\Sigma}\frac{1}{l}b_{l-1})\\
	& =d(\underset{l\geq 1}{\Sigma}\frac{1}{l}t^{l}b_{l-1}-t\underset{l\geq 1}{\Sigma}\frac{1}{l}b_{l-1})+d(t\underset{l\geq 1}{\Sigma}\frac{1}{l}b_{l-1})
	\end{align*}
	One can see that $\underset{l\geq 1}{\Sigma}\frac{1}{l}t^{l}b_{l-1}-t\underset{l\geq 1}{\Sigma}\frac{1}{l}b_{l-1}\in \Omega X$ and that $d(t\underset{l\geq 1}{\Sigma}\frac{1}{l}b_{l-1})=dt\otimes \underset{l\geq 1}{\Sigma}\frac{1}{l}b_{l-1}, $
	therefore \begin{center}
		$[\overline{x} ]= [dt\otimes \underset{l\geq 1}{\Sigma}\frac{1}{l}b_{l-1}]=H_*(\Phi)([s^{-1}\underset{l\geq 1}{\Sigma}\frac{1}{l}b_{l-1} ])$
	\end{center}
	This implies that $H_*(\Phi)$ is surjective. 
	
	To prove that $H_*(\Phi)$ is injective, let's take $[s^{-1}x] \in (red_0s^{-1}X)_{triv}$ such that $H_*(\Phi)([s^{-1}x] )=0.$ This implies that $dtx=d\overline{x},$ for a given  $\overline{x}\in \Omega X.$ As before we set $\overline{x}=\underset{l\geq 1}{\Sigma} t^{l}a_l +\underset{k\geq 0}{\Sigma} t^{k}dt b_k, $ with $\underset{l\geq 1}{\Sigma}a_l=0.$
	An easy comparison on the degree of the polynomials proves that 
	\begin{align*}
	dtx=\underset{l\geq 1}{\Sigma} lt^{l-1}dta_l+ \underset{l\geq 1}{\Sigma} t^{l}dtda_l -\underset{k\geq 0}{\Sigma} t^{k}dt db_k&\Longleftrightarrow x=a_1-db_0  \text{ and } \forall l\geq 2, a_l=\frac{1}{l}db_{l-1}\\
	&\Longrightarrow x=-\underset{l\geq 2}{\Sigma}\frac{1}{l}db_{l-1}-db_0=d(-\underset{l\geq 1}{\Sigma}\frac{1}{l}b_{l-1})
	\end{align*}
	this means that $[s^{-1}x]=0$ and  proves that $H_*(\Phi)$ is injective.
	
\end{proof}
\begin{mylemma}\label{Loop_algebra_is_infinite_loop}
	If $Y$ is an $\mathcal{O}$-algebra such that $Y\simeq \Omega X$ then $Y\simeq \Omega^{\infty} UY.$
\end{mylemma}
\begin{proof}
	From Lemma \ref{Trivial_algebra_on_Omega}, we deduce that $Y\simeq  \Omega^{\infty} s^{-1}X.$ When we apply the forgetful functor $U,$ we get the quasi-isomorphism in chain complexes $UY\simeq U \Omega^{\infty} s^{-1}X.$ We apply again the functor $\Omega^{\infty}$ and get the $\mathcal{O}$-algebra weak equivalences \begin{center}
		$\Omega^{\infty}UY \simeq \Omega^{\infty}U \Omega^{\infty} s^{-1}X \cong \Omega^{\infty} s^{-1}X \simeq Y.$
	\end{center}	
\end{proof}
In this sense, strictly speaking we will just say that any loop space in  $\text{Alg}_\mathcal{O}$ has a trivial $\mathcal{O}$-algebra structure. 
\subsection{Homotopy pushouts in $\text{Alg}_\mathcal{O}$ }
We assume in this section that the ground field $\Bbbk$ is of characteristic 0 in order to describe an explicit model for homotopy pushouts. In the particular case of describing  the suspension of $\mathcal{O}$-algebras, the ground field $\Bbbk$ can be of any characteristics.
\subsection*{Construction of a cylinder of a quasi-free  $\mathcal{O}$-algebra }
We give in this part the construction of a cylinder of a quasi-free $\mathcal{O}$-algebra in the same line that the definition for differential graded Lie algebras in \cite[II.5.]{Tanre}, and for closed DGL's in \cite[$§$ 5.]{Urtzi_Yves_Tanre_Aniceto}.  

Let $(\mathcal{O}(V), d)$ be a quasi-free $\mathcal{O}$-algebra, and let $V'$ be a copy of $V.$ We define :
\begin{enumerate}
	\item [-] $\mathcal{O}(V)\widehat{\otimes} \mathcal{I}:=(\mathcal{O}(V\oplus V'\oplus sV'), D),$ where:
	$(sv')_n=v'_{n-1},$ $Dv'=0, $ $Dsv'=v',$ $Dv=dv.$
	\item [-] $\lambda_0: (\mathcal{O}(V),d)\longrightarrow \mathcal{O}(V)\widehat{\otimes} \mathcal{I}$ the canonical injection;
	\item[-] $p: \mathcal{O}(V)\widehat{\otimes} \mathcal{I} \longrightarrow (\mathcal{O}(V),d)$ is the $\mathcal{O}$-algebra morphism given by:
	
	$p(v)=v;$ $p(v')=p(sv')=0;$ $p$ is a quasi-isomorphism since $\mathcal{O}(V'\oplus sV')$ is acyclic.
	\item[-] $i: \mathcal{O}(V)\widehat{\otimes} \mathcal{I} \longrightarrow \mathcal{O}(V)\widehat{\otimes} \mathcal{I} $ is the degree $+1$ $\mathcal{O}$-algebra derivation given by: $i(v)=sv';$ $i(sv')=i(v')=0;$
	\item[-] The $\mathcal{O}$-algebra derivation of degree $0,$ $\theta= Di+iD$ verifies $\theta D=D \theta, \theta(v')=\theta(sv')=0.$ We have the induced automorphism of $\mathcal{O}$-algebras $e^{\theta}= \underset{n\geq 0}{\Sigma} \frac{\theta^n}{n!}$ ( with inverse $e^{-\theta}$). 
\end{enumerate}
The automorphism $e^{\theta}$ is well defined for the following reason: let $v\in V_n.$ We write down explicitly the differential $d$ of $(\mathcal{O}(V),d)$ by $d=d_1+ d_2+ ...,$ where $d_kv\in \mathcal{O}(k)\underset{\Sigma_k}{\otimes}V^{\otimes k},$ for any given $k.$  Computation gives that $\theta^2 (v)= \theta i(d_2v + d_3v+...) \in    \mathcal{O}(V_{<n})\widehat{\otimes} \mathcal{I}.$ Therefore we deduce inductively that for any $x\in \mathcal{O}(V)\otimes I,$ there always exist an integer $n_x$ such that $ \theta ^{n_x}(x)=0.$

\begin{enumerate}
	\item [-] We define the second injection $\lambda_1: (\mathcal{O}(V),d)\longrightarrow \mathcal{O}(V)\widehat{\otimes} \mathcal{I}$ by, $\lambda_1 (v)=e^{\theta}(v).$
\end{enumerate}	
The couple $ (\mathcal{O}(V)\widehat{\otimes} \mathcal{I}, \lambda_0, \lambda_1, p)$ forms a cylinder of $(\mathcal{O}(V),d).$

\subsection*{Construction of homotopy pushouts in $\text{Alg}_\mathcal{O}$}

From now to short expressions we set the cobar-bar functor of $\mathcal{O}$-algebras $ (-)^c: Z \longmapsto Z^c:=B^{c}(B(\mathcal{O}),B(\mathcal{O},Z)),$ and the cylinder object defined above and associated to $Z^c$ will be denoted simply by $Z^c\widehat{\otimes} \mathcal{I}:= (\mathcal{O}(V\oplus V'\oplus sV'), D_1),$ where $V= B(\mathcal{O}, Z).$

Let $Z\overset{f}{\longrightarrow} Y$ in $\text{Alg}_\mathcal{O},$ we apply the  functor $(-)^c$ to get the weakly equivalent morphism 
$Z^c \overset{f^c}{\longrightarrow } Y^c.$  Let us consider the commutative diagram  in $\text{Alg}_\mathcal{O}:$
\begin{center}
	$
	\xymatrix{
		Z^c \ar[r]^{f^c} \ar[d]_-{ i_0}^{\simeq} &	Y^c\ar@/^/[ddrr]^{1_{Y^c}}_{=} \ar[d]^{\simeq}_{\pi_2}  \\
		Z^c \widehat{\otimes} \mathcal{I} \ar[r] ^-{\pi_1} \ar@/_/[dr]_{p} & Z^c\widehat{\otimes} \mathcal{I} \underset{Z^c}{\amalg}Y^c	 \ar@{.>}[drr]|-{f^cp\amalg 1_{Y^c}}&  & &\\
		&Z^c \ar@/_/[rr]_{f^c} 	& &Y^c 	
	}
	$
\end{center}
where the square in the middle is a pushout.  
From the lower triangle, we can then build the following  factorization of $f^c:$ 
\begin{center}
	$  \xymatrix{
		Z^c \ar[r]^-{\pi_1 i_1}  \ar@/_/[drr]_{f^c}&  Z^c \widehat{\otimes} \mathcal{I}\underset{Z^c}{\amalg}Y^c \ar[dr]_{\simeq}^-{f^c p \amalg 1_{Y^c}} & \\
		&  & Y^c}$
\end{center}
We use this later factorization to replace $f^c$ in the diagram $D^c: \xymatrix{ X^c&Z^c\ar[r]^{f^c} \ar[l]_{g^c}&Y^c } $  by the cofibration  $\pi_1i_1.$ 
\begin{proposition}\label{Homotopy_pushout}
	Given a $\mathcal{O}$-algebra diagram $D: 	 \xymatrix{ X&Z\ar[r]^{f} \ar[l]_{g}&Y }, $ a homotopy pushout of $D$ is given by
	$C_D=X^c \underset{Z^c}{\amalg} Z^c \widehat{\otimes} \mathcal{I}\underset{Z^c}{\amalg}Y^c.$ 
	Namely \begin{center}
		$C_D= \text{ colim}_{\text{Alg}_\mathcal{O} }( \xymatrix{ X^c&Z^c\ar[r]^-{ \pi_1i_1} \ar[l]_{g^c}&  	Z^c \widehat{\otimes} \mathcal{I}\underset{Z^c}{\amalg}Y^c  }).$
	\end{center}
\end{proposition}
\begin{proof}	
	This is analogue as the proof of Proposition \ref{holims}. We simply replace holims by hocolims and $Z^{\mathcal{I}}$ by $Z\widehat{\otimes} \mathcal{I}.$
	\begin{enumerate}
		\item Let $D_i: \xymatrix{ X_i&Z_i\ar[r]^{f_i} \ar[l]_{g_i}&Y_i }, i\in \{1, 2\},$ be two $\mathcal{O}$-algebra diagrams so that $D_1 \overset{\simeq}{\longrightarrow} D_2.$ One make the following computations:
		\begin{align*}
		\Sigma^{\infty}C_{D_1}\overset{(1)}{\simeq} &colim_{Ch}( UB(\mathcal{O},X_1) \overset{B(\mathcal{O},g_1 )}{\longleftarrow} UB(\mathcal{O},Z_1) \overset{  	UB(\mathcal{O}, \pi_1i_1)  }{\longrightarrow} UB(\mathcal{O}, 	Z_1^c \widehat{\otimes} \mathcal{I}\underset{Z_1^c}{\amalg}Y_1^c  ))\\
		\overset{(2)}{\simeq} & hocolim_{Ch}( UB(\mathcal{O},X_1) \overset{B(\mathcal{O},g_1 )}{\longleftarrow} UB(\mathcal{O},Z_1) \overset{  	UB(\mathcal{O}, \pi_1i_1)  }{\longrightarrow} UB(\mathcal{O}, 	Z_1^c \widehat{\otimes} \mathcal{I}\underset{Z_1^c}{\amalg}Y_1^c  ))\\
		\overset{}{\simeq} & hocolim_{Ch}( UB(\mathcal{O},X_2) \overset{}{\longleftarrow} UB(\mathcal{O},Z_2) \overset{  	 }{\longrightarrow} UB(\mathcal{O}, 	Z_2^c \widehat{\otimes} \mathcal{I}\underset{Z_2^c}{\amalg}Y_2^c  ))\\
		\simeq &\Sigma^{\infty}C_{D_2}
		\end{align*}
		where \begin{enumerate}
			\item [-] $(1)$ is obtained by applying the left adjoint functor $(-)^{ab}$(which is equivalent in this case to $\Sigma^{\infty}$) to the diagram $C_{D_1};$
			\item[-] $(2)$ is obtained by replacing colim with hocolim since $ 	UB(\mathcal{O}, \pi_1i_1) $ is an injection ( cofibration).
		\end{enumerate}
		One then obtain  $C_{D_1} \simeq C_{D_2},$ and deduce that the functor $C_-$ preserves weak equivalences of diagrams of the form $\bullet \longleftarrow \bullet \longrightarrow \bullet$ in $\text{Alg}_\mathcal{O}.$ One deduce from this property that $C_{D_{vc}}\overset{\simeq}{\longrightarrow} C_D,$ where $D_{vc}\overset{\simeq}{\longrightarrow} D$ is a virtually cofibrant replacement of $D.$
		\item Now we prove that $C_{D_{vc}}  \overset{\simeq}{\longrightarrow}  \text{colim}_{\text{Alg}_\mathcal{O}}(D_{vc}) .$ We consider $D_{vc}: X\longleftarrow Z \longrightarrow Y,$ and we form the diagram:
		\begin{center}
			\begin{tikzcd}[row sep=scriptsize, column sep=scriptsize]
				& Z \arrow[dl,  " "] \arrow[rr, "h_1"]  & & Y \arrow[dl,   " "]  \\ X \arrow[rr, crossing over , " h_2"]  & & colim _{ \text{Alg}_\mathcal{O}}D_{vc} \\
				& Z^c \arrow[dl, " "] \arrow[uu, " \simeq "] \arrow[rr , "h_4"] & & Z^c \widehat{\otimes} \mathcal{I}\underset{Z^c}{\amalg}Y^c   \arrow[uu, " \simeq "] \arrow[dl, " "] \\
				X^c \arrow[uu, "\simeq"] \arrow[rr , " h_3"] & & C_{D_{vc}} \arrow[uu, " "] \\
			\end{tikzcd}
		\end{center}
		The square obtained from the homotopy fibers of the horizontal morphisms $h_1, h_2, h_3$ and $h_4$ is described as follows: 
		\begin{center}
			$\xymatrix{
				hocofibre(h_1) \ar[r]^-{(1) }_-{\simeq}\ar[r]&	hocofibre(h_2) \\
				hocofibre(h_4)\ar[u]^{(2)}_\simeq \ar[r]_{(3)}^{\simeq}	& 	hocofibre(h_3)  \ar[u]^{(4)}_{ }
			}$      
		\end{center}	
		where 
		\begin{enumerate}
			\item [(1)]  is a weak equivalence as the top square of the cube is a homotopy pushout;
			\item[(2)] is a weak equivalence because the back face of the cube is trivially a homotopy pushout;
			\item[(3)] is a weak equivalence. In fact the functor $\Sigma^{\infty}(-)$ applied to the diagram at the bottom (of cofibrant algebras) gives a homotopy pushout diagram  in $Ch.$ One then deduce that 
			\begin{center}
				$\Sigma^{\infty}	hocofibre(h_4) \overset{\simeq}{ \longrightarrow} \Sigma^{\infty}	hocofibre(h_3)$
			\end{center}
			and equivalently we get 
			\begin{center}
				$	hocofibre(h_4) \overset{\simeq}{ \longrightarrow} 	hocofibre(h_3)$
			\end{center}
		\end{enumerate}
		By this we conclude that $(4)$ is a weak equivalence, and therefore that $C_{D_{vc}}  \overset{\simeq}{\longrightarrow}  \text{colim}_{\text{Alg}_\mathcal{O}}(D_{vc}) .$

	\end{enumerate}
	
\end{proof}
\begin{remark}
	If $D: 	 \xymatrix{ X&Z\ar[r]^{f} \ar[l]_{g}&Y }$ is a diagram of quasi-free $\mathcal{O}$-algebras, then we don't need the cofibrant replacement functor $(-)^{c}$ in the construction, and we have 
	\begin{center}
		$C_D=X \underset{Z}{\amalg} Z\widehat{\otimes} \mathcal{I}\underset{Z}{\amalg}Y.$ 
	\end{center}

\end{remark}

\begin{mylemma}\label{Proposition The suspension of free algebra}
	Let $(\mathcal{O}(V),d)$ be a quasi-free algebra with the notation for the differential:  $d=d_1+d_2+ ...$.
	Then $\Sigma (\mathcal{O}(V),d) \simeq (\mathcal{O}(sV'), D_1),$ where $D_1(sv'):=-sd_1v'$ and $V'$ is a copy of $(V,d).$
	
\end{mylemma}
\begin{proof} We set for short $Z=(\mathcal{O},d);$
	
	In  Proposition \ref{Homotopy_pushout}, we have proved that $(0 \underset{Z}{\amalg} Z \widehat{\otimes} \mathcal{I}\underset{Z}{\amalg}0 , D) \simeq \Sigma Z.$
	
	Since $(e^{\theta})^{ab}(v)= v'+sd_1v',$ we deduce that in $(0 \underset{Z}{\amalg} Z \widehat{\otimes} \mathcal{I}\underset{Z}{\amalg}0)^{ab},$
	\begin{align*}
	[Dsv'] &=[v']\\
	&= [v'+sd_1v-sd_1v'] \\
	&= [-sd_1v']
	\end{align*}
	Now we consider the morphism of $\mathcal{O}$-algebras  
	\begin{align*}
	\psi: (\mathcal{O}(sV'), D_1) & \longrightarrow (0 \underset{Z}{\amalg} Z \widehat{\otimes} \mathcal{I}\underset{Z}{\amalg}0, D)
	\end{align*}
	given by $ \psi(sv')= [sv']. $
	
	This is a well defined chain complex morphism since $[D\psi(sv')] = \psi (D_1(sv'))$ and in addition $B(\mathcal{O}, \psi)\simeq \psi^{ab}$ is a quasi-isomorphism. We deduce that $\psi$ is a quasi-isomorphism. 
\end{proof}
	\begin{remark}\label{Remark The suspension of free algebra}
	The result of Lemma \ref{Proposition The suspension of free algebra} holds in general when the ground field $\Bbbk$ is of any characteristics. In fact, we have the following pushout diagram 
	\begin{center}
		$\xymatrix{ \mathcal{O}(V) \ar[r]\ar[d]& \mathcal{O}(V\oplus sV)\simeq 0\ar[d]\\
			0\ar[r]& \mathcal{O}(sV)
		}$
	\end{center}
	This  is also a homotopy pushout diagram, thus we deduce that $\Sigma \mathcal{O}(V)\simeq  \mathcal{O}(sV).$
	
\end{remark}
\begin{corollary}\label{Suspension_algebra_is_free}We assume that the ground field $\Bbbk$ is of any characteristics.
	Given an $\mathcal{O}$-algebra $Z,$ then 
	$\Sigma Z$ is the free   $\mathcal{O}$-algebra    $(\mathcal{O} (sUB(\mathcal{O},Z)), d_1),$ 
	where $d_1$ is the internal differential induced by the differential of $(B(\mathcal{O},Z),d).$ 
\end{corollary}
\begin{proof}
	We make the following computation
	\begin{align*}
	\Sigma Z & \simeq \Sigma B^{c}(B(\mathcal{O}),B(\mathcal{O},X)) \ \ \  \ \ \ (\text{Using Thm \ref{GJ}  and Remark \ref{Remark cobar-bar duality in any characteristics}})\\
	& \simeq \mathcal{O} (sUB(\mathcal{O},Z))  \ \ \  \ \ \ (\text{ Using Proposition \ref{Proposition The suspension of free algebra} and Remark \ref{Remark The suspension of free algebra} })
	\end{align*}
\end{proof}
Finally, we remind the following relation between holims and hocolims.
\begin{mylemma}\label{Lemma Filtered hocolim commutes with finite limits}
	In $\text{Alg}_\mathcal{O},$ filtered homotopy colimits, that are colimits of  filtered diagrams, commute with finite homotopy limits.
\end{mylemma}	
\begin{proof}
	This follows from the fact that this property is true in $Ch_+,$ and that the forgetful functor $U: \text{Alg}_\mathcal{O} \longrightarrow Ch_+$ commutes with finite limits and filtered colimits.
\end{proof}

\section{Goodwillie approach in functor calculus}\label{Section Goodwillie Calculus}
We assume in that section that the ground field $\Bbbk$ is of any characteristics.
Let  $\mathcal{C}$ and $\mathcal{D}$ be any of the model categories $\text{Alg}_\mathcal{O}, Ch_+$ or $Ch.$  We remind in this section the theory of functor calculus, for functors $F: \mathcal{C} \longrightarrow \mathcal{D}.$ We follow the lines of \cite[$§$ 4 and $§$ 5]{Kuhn07}( and implicitly \cite{GIII03} ), except that in our case our functors are not required to be continuous. This was a way for Kuhn to get an assembly map $F(X)\otimes K \longrightarrow F(X\otimes K),$ where $K \in sSets,$ and $X \in \mathcal{C}.$ In fact,  we will see (in Lemma \ref{Multilinear_assembly_map}) that homotopy functors $F: Ch\longrightarrow Ch$ have a natural assembly map (at least at the level of the homotopy category $HoCh$).
However, we will require our functors to be homotopy(preserve weak equivalences), which is a weaker version of being continuous, since continuous implies homotopy.

\begin{definition}[Homotopy functor]
	Let $\mathcal{C}$ and $\mathcal{D}$ be any of the model categories $\text{Alg}_\mathcal{O},$   $ Ch_+,$ or $Ch$ and $F:\mathcal{C}\longrightarrow \mathcal{D}$ be a functor.
	\begin{enumerate}
		\item The functor $F$ is reduced if $F(0)\simeq 0;$
		\item $F$ is a  homotopy functor if it preserves weak equivalences.
		\item $F$ is finitary if it preserves filtered homotopy colimits.
	\end{enumerate}
	
\end{definition}

\begin{definition}[$n$-excisive functor]
	Let $\mathcal{C}$ and $\mathcal{D}$ be any of the model categories $\text{Alg}_\mathcal{O},$   $ Ch_+$ or $Ch$ and $F:\mathcal{C}\longrightarrow \mathcal{D}$ be a homotopy functor.
	\begin{enumerate} \item An $n$-cube in $\mathcal{C}$ is a functor $\mathcal{X}: \mathcal{P}(\underline{n})\longrightarrow \mathcal{C},$ where $\mathcal{P}(\underline{n})$
		is the poset of subsets of $\underline{n}:=\{1, ..., n\}.$

		\item The  functor $F:\mathcal{C}\longrightarrow \mathcal{D}$ is called $n$-excisive if whenever $\mathcal{X}$ is a strongly coCartesian $n+1$-cube in $\mathcal{C}, F(\mathcal{X})$ is a cartesian cube in $\mathcal{D}.$ 
		
	\end{enumerate}		
\end{definition}

\begin{definition}
	[\cite{Kuhn07},  4.6] 	Let $\mathcal{C}$ be any of the model categories $\text{Alg}_\mathcal{O},$ $ Ch_+$ or $Ch.$  
	Let $X\in \mathcal{C}$ and $T$ be a finite set. We define the joint $X*T,$  of $X$ and $T, $ to be the homotopy cofiber of the folding map\begin{align*}
	X*T=&\text{hocof }(\underset{T}{\amalg}X\overset{\bigtriangledown }{\longrightarrow}X)
	\end{align*}
\end{definition}
\begin{example}Using Proposition \ref{Homotopy_pushout}, we make the following computation:
	for $X\in \text{Alg}_\mathcal{O},$ 
	\begin{enumerate}
		\item [-]$X*\underline{0}= X*\emptyset=B^c(B(\mathcal{O}), B(\mathcal{O}, X));$
		\item [-]$X*\underline{1}=cB^c(B(\mathcal{O}), B(\mathcal{O}, X));$
		\item [-] $X*\underline{2}=\Sigma B^c(B(\mathcal{O}), B(\mathcal{O}, X)).$
	\end{enumerate}

\end{example}
Let $\mathcal{C}$ and $\mathcal{D}$ be any of the model categories $\text{Alg}_\mathcal{O},$ $ Ch_+$ or $Ch$ and let $F: \mathcal{C}\longrightarrow  \mathcal{D}$ be a homotopy and reduced functor.   For $X\in \mathcal{C},$ define the $n$-cube 
\begin{center}
	$\chi_n(X): \mathcal{P}(n)\longrightarrow\mathcal{C}$ by $\chi_n(X):T\longmapsto X*T.$
\end{center}
This is a strongly coCartesian $\underline{n}$-cube (see \cite[lemma 7.1.4]{Walt06}), the fact is that homotopy colimits commute with themselves. 	One set 
\begin{center}
	$T_{n-1}F(X):=\underset{T\in \mathcal{P}(\underline{n})-\{\emptyset\}}{holim} F(\chi_n(X)(T))$
\end{center}
If $F$ is $n$-excisive, then the natural map 
$t_{n-1}F: F(X)=F( \chi_n(X)(\emptyset))\longrightarrow T_{n-1}F(X)$
is a weak equivalence.
Write $T^{i}_{n-1}F$ defined inductively by $T^{i+1}_{n-1}F:=T_{n-1}(T^{i}_{n-1}F)$
and  \begin{center}
	$P_{n-1}F:=\text{hocolim }(F\overset{t_{n-1}F}{\longrightarrow} T_{n-1}F\overset{T_{n-1}(t_{n-1}F)}{\longrightarrow} T_{n-1}(T_{n-1}F) \overset{T^{2}_{n-1}(t_{n-1}F)}{\longrightarrow}...) $
\end{center}
\begin{example}\label{P_1 approximation}$ T_1F(X)=\text{holim }(F(X*\underline{1})\longrightarrow F(X*\underline{2})\longleftarrow F(X*\underline{1}));$
	if $F$ is reduced then $F(X*\underline{1})\simeq 0$ and we deduce  that $T_1F(X)\simeq \Omega F(\Sigma X);$
	Therefore inductively we get
	\begin{center}
		$P_1F(X)\simeq \underset{p\rightarrow \infty}{\text{hocolim } }\Omega^{p}F\Sigma^{p}$
	\end{center}
	
\end{example}
\begin{definition}[homogeneous functors]
	Let $F:\mathcal{C}\longrightarrow \mathcal{D}$ be a homotopy and reduced functor. $F$ is called $n$-homogeneous if \begin{enumerate}
		\item [-] $F$ is $n$-excisive and 
		\item[-] $P_{n-1}F\simeq 0.$
	\end{enumerate}
\end{definition}

When $\mathcal{D}= \text{Alg}_\mathcal{O},$ we make the following remark:

\begin{remark}\label{trivial_algebra_structure} Let $\mathcal{C}=\text{Alg}_\mathcal{O} \text{ , } Ch_+$ or $Ch.$
	If a functor $F:\mathcal{C}\longrightarrow \text{Alg}_\mathcal{O}$ is $n$-homogeneous, then for any $X\in \mathcal{C}, F(X)$ has a trivial $\mathcal{O}$-algebra structure. In fact Goodwillie \cite[ Lemma 2.2]{GIII03} proves  in a completely general argument that there is a homotopy pullback diagram 	\begin{center}
		$\xymatrix{
			P_nF \ar[r]^{} \ar[d]_{}  & 	P_{n-1}F \ar[d]^{}\\
			0 \ar[r]^{}& 	R_nF}$  
	\end{center}
	, where $R_nF: \mathcal{C}\longrightarrow \text{Alg}_\mathcal{O}$ is $n$-homogeneous.
	Thus if $F$ is $n$-homogeneous, then $F\simeq P_nF\simeq \Omega R_nF.$  Therefore, when the ground field is of characteristics 0, we can rewrite $F$ as $F\simeq \Omega^{\infty}UF,$ where $U: \text{Alg}_\mathcal{O}\longrightarrow Ch_+$ is the forgetful functor (see Lemma \ref{Loop_algebra_is_infinite_loop}).
\end{remark}

Since $F=T^{0}_{n-1}F,$ the functor $P_{n-1}F$ is equipped with a map $F\longrightarrow P_{n-1}F.$ In addition the inclusion of categories $\mathcal{P}(\underline{n})\longrightarrow \mathcal{P}(\underline{n+1})$ induces a map $T_nF\longrightarrow T_{n-1}F$ which extends formally to give a map $q_nF: P_nF\longrightarrow P_{n-1}F$ which is a fibration (see \cite[Page 664]{GIII03}). By inspection this map is again a fibration in $\text{Alg}_\mathcal{O}$  and in $ Ch_+,$ since the maps $T_nF\longrightarrow T_{n-1}F$ will always be a surjection, and filtered colimits of surjections is again a surjection.
\begin{theorem}\cite[1.13]{GIII03}
	A homotopy functor $F:\mathcal{C}\longrightarrow \mathcal{D}$ determines a tower of functors $\{ P_nF:\mathcal{C}\longrightarrow \mathcal{D}\}_n,$		
	where $P_nF$ are $n$-excisif, $q_nF:P_nF\longrightarrow P_{n-1}F$ are fibrations, the functors	$D_nF=\text{ hofibre }(q_nF)$ are $n$-homogeneous.
	
\end{theorem}

\begin{remark}\label{Remark Pn preserves finite homotopy limits}
	A straight consequence of Lemma \ref{Lemma Filtered hocolim commutes with finite limits} for this section is that the functor $P_n,$ which is basically a homotopy colimit, commutes with finite homotopy limits. In particular $P_n$ preserves fiber sequences.	
\end{remark}

\section{Characterization of homogeneous functors}\label{Section Characterisation of homogeneous functors}
In this section, we characterize homogeneous functors with the cross effect. Before getting to this result, we will make a couple of constructions and provide intermediate results. The characterization itself appears in Corollary \ref{Derivatives} at the end of this section.

There are two ways to define the cross effect associated to a functor. One can define it as a homotopy fiber (hofib) and we can also define it as a total homotopy fiber (thofib). These definitions are reported here bellow.

\begin{definition}[Cross-effects]
	Let $\mathcal{C}$ and $\mathcal{D}$ be any of the model categories $\text{Alg}_\mathcal{O},$ $ Ch_+$ or $Ch$ and let $F: \mathcal{C}\longrightarrow  \mathcal{D}$ be a homotopy and reduced functor.
	We define $cr_nF: \mathcal{C}^{\times n}\longrightarrow \mathcal{D},$ the $n^{th}$ cross-effect of $F,$ to be the functor of $n$ variables  given by 
	
	\begin{center}
		$ cr_nF(X_1, ..., X_n)=\text{ hofib}\{F(\underset{i\in \underline{n}}{ \amalg}X_i)\longrightarrow \underset{T\in \mathcal{P}_{0}(n)}{holim} F( \underset{i\in \underline{n-T}}{ \amalg} X_i ) \}$
	\end{center}
\end{definition}
This is equivalent to define the $n^{th}$ cross-effect of $F$ as:
\begin{center}
	$cr_nF(X_1, ..., X_n)=\text{ thofib} (T\supseteq \underline{n}\mapsto F( \underset{i\in \underline{n}-T}{\amalg}X_i)).$
\end{center}

In the particular cases where $\mathcal{C}=Ch_+$ or $Ch$ and $\mathcal{D}=Ch,$ we can also describe the cross effect of a functor $F$ using the total homotopy cofiber (thocofib)\index{thocofib: Total homotopy cofiber} of a certain cube. This dual construction, also called the "co-cross-effect",  was considered by McCarthy \cite[$1.3$]{Randy}  in studying dual calculus, and the equivalence  between the cross-effect and co-cross-effect was proved by Ching \cite[Lemma $2.2$]{Ching10}  for functors with values in spectra. Let  $W_1, ..., W_n \in \mathcal{C},$ we associate the  $n$-cube $\mathcal{X}$ in $\mathcal{C}$ defined as follows: 

\begin{enumerate}
	\item [-] $T\subseteq \underline{n}, \mathcal{X}(T):=\underset{i\in T}{\oplus} W_j;$
	\item[-] For $T \subsetneq \underline{n}$ and $j\in \underline{n}\backslash T,$ the map $\mathcal{X}(T) \longrightarrow \mathcal{X}(T\cup \{j\})$ (in the cube) is induced by the inclusion
	\begin{align*}
	\underset{i\in T}{\oplus}W_i &\longrightarrow ( \underset{i\in T}{\oplus}W_i) \oplus W_j\\
	x &\longmapsto (x,0)
	\end{align*}
\end{enumerate}
\begin{definition}[Co-cross-effects]
	Let $\mathcal{C}=Ch_+$ or $Ch$ and $F: \mathcal{C}\longrightarrow Ch$ be a homotopy functor. The $n^{th}$ co-cross	effect of F is the functor $cr^nF: \mathcal{C}^{\times n} \longrightarrow Ch$ which computes the homotopy total fiber of $F(\mathcal{X}).$ That is:
	\begin{center}
		$cr^nF(W_1, ..., W_n):=\text{ hocofib}\{ \underset{T\subsetneq \underline{n}}{hocolim} F( \underset{i\in T}{ \oplus} W_i ) \longrightarrow F(W_1\oplus ... \oplus W_n)\}.$
	\end{center}
\end{definition}

\begin{mylemma}\label{Lemma Cross-effect equal Co-Cross-effect}
	Let $\mathcal{C}=Ch_+$ or $Ch$ and $F: \mathcal{C}\longrightarrow Ch$ be a homotopy functor. Then the $n^{th}$ cross-effect of $F$ is equivalent to the $n^{th}$ co-cross-effect of $F.$ That is:
	\begin{center}
		$cr_nF(W_1, ..., W_n)   \overset{\simeq}{\longrightarrow} cr^nF(W_1, ..., W_n) $
	\end{center}
\end{mylemma}
\begin{proof}
	Since $Ch$ is a stable category and that in $\mathcal{C}$ products and coproducts are isomorphic, we simply mimic Ching's proof.
\end{proof}

To understand homogeneous functors, Goodwillie\cite{GIII03} pointed the following proposition for functors with values in spectra. We reformulate it in our algebraic context though the proof follows literally [ \cite{GIII03}, proposition 3.4] and [\cite{GII92}, proposition 2.2].

\begin{proposition}\label{Cross_equival}Let $\mathcal{C}$ and $\mathcal{D}$ be any of the model categories $\text{Alg}_\mathcal{O},$ $ Ch_+$ or $Ch.$
	If $H: \mathcal{C}\longrightarrow \mathcal{D}$ is an $n-$excisive and reduced functor such that  $cr_nH\simeq 0$, then $H$ is $(n-1)-$excisive.
\end{proposition}
\begin{proof}
	
	\begin{enumerate}
		\item [(i)]One define the $n$-cube $\mathcal{X}=\mathcal{S}^{*}(X_1, ..., X_n),$ for objects $X_1, ..., X_n$ in $\mathcal{C},$ as follows: $\forall T\subseteq [n],  \mathcal{X}(T):=\underset{i \in T}{\amalg}X_i$ and $\mathcal{X}(\emptyset)=0.$
The maps in the cube $\mathcal{X}$ are inclusions. We 
		associate to this cube $\mathcal{X}$ the $n$-cube $\mathcal{S}(X_1, ..., X_n)$ which has the same objects with $\mathcal{X},$ but where the inclusions are reversed to the projections. 
		Let $U: \mathcal{D}\longrightarrow Ch$ be the forgetful functor when $\mathcal{D}=\text{Alg}_\mathcal{O}$ and be the identity functor when $\mathcal{D}=Ch_+.$ We make the following computations:
		\begin{align*}
		U	\text{ thofib }H(\mathcal{X})\cong\text{ thofib }UH(\mathcal{X})&=\text{ thofib }UH(\mathcal{S}^{*}(X_1, ..., X_n))\\
		&= \Omega^{n}\text{ thofib }UH(\mathcal{S}(X_1, ..., X_n))\\
		&=\Omega^{n}cr_n(UH)(X_1, ..., X_n) \\
		&=\Omega^{n}Ucr_nH(X_1, ..., X_n)\simeq 0,
		\end{align*}
		One will then conclude from these that 
		$	\text{ thofib }H(\mathcal{X})\simeq 0$( or equivalently that $H(\mathcal{X})$ is cartesian) for all strongly coCartesian cubes $\mathcal{X}$ in which $\mathcal{X}(\emptyset)=0,$ since any such cube $\mathcal{X}$ is naturally equivalent to $\mathcal{S}^{*}(\mathcal{X} (\{1\}), ..., \mathcal{X} (\{n\}))$(see \cite[proposition 2.2]{GII92}).  
		\item[(ii)] 	Let $\forall T\subseteq [n],$ and $a,b\in [n].$ Given an arbitrary strongly coCartesian n-cube $\mathcal{X}$ in $\mathcal{C},$ put $\mathcal{X}'(T)=hocolim(0\longleftarrow \mathcal{X}(\emptyset)\longrightarrow \mathcal{X}(T)).$
	
		We have the following commutative diagram	
		\begin{center}
			$\xymatrix{
				\mathcal{X}(\emptyset) \ar[r]^{}\ar[d]^{} &	\mathcal{X}(T) \ar[d]^{ } \ar[r] & \mathcal{X}(T\cup \{a\})\ar[d]\\
				0 \ar[r]^{} &	\mathcal{X}'(T)  \ar[r] & \mathcal{X}'(T\cup \{a\}) }$        
		\end{center}
		where the largest square is a homotopy pushout along with the most left square.	It then follows that the most right square is also a homotopy pushout and therefore that the following square is a homotopy pushout:
		\begin{center}
			$\xymatrix{
				\mathcal{X}(T) \ar[r]^{}\ar[d]^{} &	\mathcal{X}(T\cup \{a\}) \ar[d]^{ } \\
				\mathcal{X}'(T) \ar[r]^{}& \mathcal{X}'(T\cup \{a\})  }$        
		\end{center}
		and therefore it follows that
		\begin{center}
			$\xymatrix{
				\mathcal{X}'(T) \ar[r]^{}\ar[d]^{} &	\mathcal{X}'(T\cup \{a\}) \ar[d]^{ } \\
				\mathcal{X}'(T\cup \{b\}) \ar[r]^{}& \mathcal{X}'(T\cup \{a,b\})  }$        
		\end{center}
		is a homotopy pushout diagram. This proves that
		the n-cube $\mathcal{X}'$ is strongly coCartesian and that the map $\mathcal{X}\longrightarrow \mathcal{X}'$ is a strongly cocartesian $n+1$-cube. $H$ is $n$-excisive, thus $H(\mathcal{X}) \longrightarrow  H(\mathcal{X}')$ is cartesian. In addition since $\mathcal{X}'(\emptyset)=0,$ we deduce from $(i)$ that   $H(\mathcal{X}')$ is cartesian and conclude that $H(\mathcal{X})$ is also cartesian.

	\end{enumerate}
	
\end{proof}

We get the following consequence:
\begin{corollary}\label{Cross_equi_Corro}
	Let $F$ and $G$ be two $n-$homogeneous functors $ \mathcal{C}\longrightarrow \mathcal{D} $, where $\mathcal{C}$ and $\mathcal{D}$ are any of the model categories $\text{Alg}_\mathcal{O},$ $ Ch_+$ or $Ch,$
	and a natural transformation  $F\overset{J}{\longrightarrow} G.$   If $ cr_n(J): cr_nF\overset{}{\longrightarrow} cr_nG$ is an equivalence, then so is $J.$  
\end{corollary}

\begin{proof}
	Let $H=\text{ hofib}(F\overset{J}{\longrightarrow}  G ).$ $H$ is $n-$homogeneous and then $n-$excisive. By hypothesis $cr_nH\cong \text{ hofib}(cr_nF\overset{cr_nJ}{\longrightarrow}  cr_nG )\simeq 0.$
	The functor $H$ gathers then the hypothesis of  Proposition \ref{Cross_equival}, thus $H$ is $n-1-$excisive . Hence we get
	\begin{center}
		\begin{tabular}{clllccc}
			$H\simeq P_{n-1}H$ &   $=$ hofib ( & $P_{n-1}F$ &  & $\longrightarrow $& &$P_{n-1}G )=0 $ \\
			&  & $\shortparallel$& &  && $\shortparallel $\\
			&  & $0$ &  &  & &$0$ 
		\end{tabular}	
	\end{center}
	One deduce from the long exact sequence obtained from the homotopy fiber sequence of $J$ that $J$ is a weak equivalence. 
	
\end{proof}

	\begin{definition}\label{diagonal_functor}
	Let $\mathcal{C}$ and $\mathcal{D}$ be any of the model categories $\text{Alg}_\mathcal{O},$ $ Ch_+$ or $Ch,$ and $F:\mathcal{C}\longrightarrow \mathcal{D}$ be a homotopy and reduced functor.
	\begin{enumerate}
		\item 	The functor  $L_nF:\mathcal{C}^n \longrightarrow \mathcal{D}$ is obtained from $cr_n F$ by 
		\begin{center}
			$L_nF(X_1, ...,X_n)\simeq \underset{p_i \rightarrow \infty}{\text{hocolim }}\Omega^{p_1 +...+p_n}cr_n F(\Sigma ^{p_1}X_1, ..., \Sigma ^{p_n}X_n)$
		\end{center}
		In the case that $\mathcal{D}=\text{Alg}_\mathcal{O},$  this filtered homotopy colimit  can be seen as a homotopy colimit in the underlying category of chain complexes.

		\item The functor $\triangle_nF: \mathcal{C} \longrightarrow \mathcal{D}$ is obtained from $L_nF$ by:
		\begin{center}
			$\triangle_nF=(L_nF)\circ \triangle$
		\end{center}where $\triangle: \mathcal{C}\longrightarrow \mathcal{C}^{\times n}$ is the diagonal map.
		The symmetric group $\Sigma_n$ acts on $\triangle_nF$ by permuting its $n$ entries of the cross effect $cr_nF.$ 
		\item The functor $\widehat{\triangle}_nF(X):  \mathcal{C} \longrightarrow Ch$ is obtained from $\triangle_nF$ by dropping the functor $red_0.$ Namely,
		\begin{center}
			$ \widehat{\triangle}_nF(X):=\underset{p_i \rightarrow \infty}{\text{hocolim }}s^{-p_1 -...-p_n}cr_n (UF)(\Sigma ^{p_1}X, ..., \Sigma ^{p_n}X)$
		\end{center}
		where $U: \text{Alg}_\mathcal{O} \longrightarrow Ch$ is the forgetful functor and this colimit is taken in the category $Ch.$ 	The symmetric group $\Sigma_n$ acts on $\widehat{\triangle}_nF(X)$ by permuting its $n$ entries of the cross effect $cr_nUF.$ 
	\end{enumerate}
	
\end{definition}

\begin{remark}

	The functor $L_nF$ of Definition \ref{diagonal_functor} can also be seen as the stabilization of the cross effect, that  is the functor obtained by applying the first Taylor approximation functor $P_1$ to each variable position of the multi-variable functor $cr_nF.$ For instance, 
	\begin{enumerate}
		\item $L_1F=P_1F$ (see Example \ref{P_1 approximation});
		\item $L_2F(X,Y)=P_1(Y\longmapsto P_1(X\longmapsto cr_2(X,Y)));$  
		\item and so on.
	\end{enumerate}
\end{remark}
We assume from now, when it is not specified, that the ground field $\Bbbk$ is of characteristic 0.

\begin{mylemma}\label{Triangle_as_Chain_complex}
	Let $\mathcal{C}$  be any of the model categories $\text{Alg}_\mathcal{O},$ $ Ch_+$ or $Ch,$ and $F:\mathcal{C}\longrightarrow \text{Alg}_\mathcal{O}$  be a homotopy and reduced functor. Then for any $X\in \mathcal{C},$ there is a weak equivalence of $\mathcal{O}$-algebras \begin{center}
		$\triangle_nF(X)\simeq (red_0 \widehat{\triangle}_nF(X))_{triv}.$
	\end{center}
\end{mylemma}
\begin{proof}
	If $U: \text{Alg}_\mathcal{O}\longrightarrow Ch$ denotes the forgetful functor,  we make the following computation:
	\begin{align*}
	U\triangle_nF(X) &\simeq \underset{p_i \rightarrow \infty}{\text{ hocolim}_{Ch }}[red_0 s^{-p_1 -...-p_n}cr_nU F(\Sigma ^{p_1}X, ..., \Sigma ^{p_n}X)]\\
	&	\simeq  red_0 \underset{p_i \rightarrow \infty}{\text{ hocolim}_{Ch}} [s^{-p_1 -...-p_n}cr_n UF(\Sigma ^{p_1}X, ..., \Sigma ^{p_n}X)]
	\end{align*}
	This last equivalence is justified by the fact that  the functor $red_0$ commutes with filtered colimits.
	Now by applying the functor $( -)_{triv},$ we get the weak equivalence of $\mathcal{O}$-algebras \begin{center}
		$ (	U\triangle_nF(X))_{triv} \simeq (red_0 \widehat{\triangle}_nF(X))_{triv}.$
	\end{center}
	In addition since the functor $	\triangle_nF $ is $n$-homogeneous, we know from Remark \ref{trivial_algebra_structure} that 	$ 	\triangle_nF(X) \simeq  (	U\triangle_nF(X))_{triv},$ therefore the result follows.
\end{proof}

We are now ready to state the next theorem which was inspired by  \cite[Thm 5.12]{Kuhn07} for functors with values in stable model categories.

\begin{theorem}\label{Nick_Kuhn} Let $\mathcal{C}$ and $\mathcal{D}$ be any of the categories  $\text{Alg}_\mathcal{O},$  $Ch_+$ and $Ch,$ and $F: \mathcal{C}\longrightarrow \mathcal{D}$ be a homotopy and  reduced functor. Then there is a weak equivalence 
	\begin{center}
		$D_nF(X)\simeq \Omega^{\infty}(\widehat{\triangle}_nF(X)_{h\Sigma_n}).$  
	\end{center}	
	where $(-)_{h\Sigma_n}$ denotes the homotopy orbits. When $\mathcal{D}=Ch_+$ or $Ch$ then this result holds when the ground field $\Bbbk$ is of any characteristics.
\end{theorem}
To prove this, we need the following lemma.

\begin{mylemma}\label{map_alpha} Let $\mathcal{C}$ be either  $\text{Alg}_\mathcal{O},$  $Ch_+$ or $Ch,$ and  $F: \mathcal{C}\longrightarrow \mathcal{D}$ be a homotopy and reduced functor. Then we have a weak equivalence  \begin{center}
		$P_n(L_nF\circ \triangle) \simeq L_n(P_nF)\circ \triangle $
	\end{center}
	
\end{mylemma}

\begin{proof} 
	One make the following observation:
	\begin{align*}
	T_n(L_nF\circ \triangle)(X):=& \underset{T\in \mathcal{P}_0(\underline{n+1})}{\text{holim }}\underset{p_i\rightarrow \infty}{\text{hocolim }} \Omega^{p_1+...+p_n} cr_nF( \Sigma^{p_1}(X*T), ... ,  \Sigma^{p_n}(X*T))\\
	\overset{(1)}{\simeq} &  \underset{T\in \mathcal{P}_0(\underline{n+1})}{ \text{holim }}\underset{p_i\rightarrow \infty}{\text{hocolim } } \Omega^{p_1+...+p_n} cr_nF( (\Sigma^{p_1}X)*T, ... ,  (\Sigma^{p_n}X)*T) \\
	= &  \underset{T\in \mathcal{P}_0(\underline{n+1})}{ \text{holim }}\underset{p_i\rightarrow \infty}{\text{hocolim }} \Omega^{p_1+...+p_n} \text{ thofib} (A\supseteq \underline{n}\mapsto F( \underset{\underline{n}-A}{\amalg}((\Sigma^{p_j}X)*T))\\
	\overset{(2)}{\simeq} &  \underset{T\in \mathcal{P}_0(\underline{n+1})}{ \text{holim }}\underset{p_i\rightarrow \infty}{\text{hocolim }} \Omega^{p_1+...+p_n} \text{ thofib}(A\supseteq \underline{n}\mapsto F(( \underset{\underline{n}-A}{\amalg}\Sigma^{p_j}X)*T)\\
	\overset{(3)}{\simeq} & \underset{p_i\rightarrow \infty}{\text{hocolim } } \Omega^{p_1+...+p_n} \text{ thofib}(A\supseteq \underline{n}\mapsto T_nF( \underset{\underline{n}-A}{\amalg}\Sigma^{p_j}X))\\
	=& \underset{p_i\rightarrow \infty}{\text{hocolim }} \Omega^{p_1+...+p_n} cr_n( T_nF)( \Sigma^{p_1}X, ..., \Sigma^{p_n}X)\\
	=& L_n(T_nF)\circ \triangle(X)
	\end{align*}
	where \begin{enumerate}
		\item [(1)] is due to the isomorphism $ \Sigma^{p_j}(X*T)\cong  (\Sigma^{p_j}X)*T,$ for each $j;$ 
		\item[(2)] is due to the isomorphism $  \underset{\underline{n}-T}{\amalg}(\Sigma^{p_j}X*T)\cong  (\underset{\underline{n}-T}{\amalg}\Sigma^{p_j}X)*T,$ for each $T\subseteq \underline{n};$  
		\item[(3)] is because finite holims commute with filtered colimits (see Lemma \ref{Lemma Filtered hocolim commutes with finite limits}), and holims commute with loops $\Omega$ and total fibers.
	\end{enumerate}
	One also deduce from this observation steps that the following square is commutative
	\begin{center}
		$\xymatrix{
			L_nF\circ \triangle \ar[r]^{=}\ar[d]^{t_nL_nF\circ \triangle} &	L_nF\circ \triangle \ar[d]^{ L_nt_nF\circ \triangle } \\
			T_n(L_nF\circ \triangle)\ar[r]^{\simeq}& L_n(T_nF)\circ \triangle }$        
	\end{center}
	Thus we can deduce by induction on the iterations from this square that 
	\begin{center}
		$P_n(L_nF\circ \triangle) \simeq (L_nP_nF)\circ \triangle.$
	\end{center}	
\end{proof}

\begin{Proof of theorem} Let $F: \mathcal{C} \longrightarrow \mathcal{D}$ be a homotopy and reduced functor.
	Let $J$ be the composition in    $ Ch:$ \begin{center}
		$((cr_nUD_nF)\circ \triangle(X))_{h\Sigma_n}\longrightarrow (UD_nF(\underset{\underline{n}}{\amalg}X))_{h\Sigma_n}\longrightarrow UD_nF(X)$
	\end{center}
	where the first map is the projection, the second map is induced by the folding map $\underset{\underline{n}}{\amalg}X \overset{\nabla}{\longrightarrow} X,$ and if $\mathcal{D}=\text{Alg}_\mathcal{O},$ then $U: \text{Alg}_\mathcal{O} \longrightarrow Ch$ is the forgetful functor, and $U$ is simply the identity functor when $\mathcal{D}=Ch;$
	Since we want to prove that $J$ is a quasi-isomorphism, we will simply show that $cr_nJ$ is a quasi-isomorphism  and conclude using Corollary \ref{Cross_equi_Corro}.	
	
	For the sake of simplicity we set $L(X)=cr_n(UD_nF)(X, ...,X)_{h\Sigma_n}.$
	\begin{align*}
	cr_nL(X_1, ..., X_n)&=thofib(L\circ S(X_1, ..., X_n)) \\
	& =thofib(\underline{n}-T\mapsto L(\underset{T}{\amalg}X_i ))\\
	&=thofib(\underline{n}-T\mapsto cr_nD_nF(\underset{T}{\amalg}X_i,..., \underset{T}{\amalg}X_i  )_{h\Sigma_n})\\
	&=thofib(\chi)_{h\Sigma_n},
	\end{align*}
	where 	$\chi: \underline{n}-T\mapsto cr_n(UD_nF)(\underset{T}{\amalg}X_i,..., \underset{T}{\amalg}X_i  ).$ Since $cr_n(UD_nF)$ is multilinear, we deduce the weak equivalence (natural in $T$)
	\begin{align}
	\chi(\underline{n}-T)&\overset{\simeq}{\longrightarrow}\underset{\pi:\underline{n}\rightarrow T }{\prod}cr_n(UD_nF)(X_{\pi(1)}, ..., X_{\pi(n)} ).
	\end{align}
	Let's consider the map $\pi:\underline{n}\rightarrow \underline{n}$ and consider the cube $\mathcal{Y}_{\pi}$ defined by: \begin{center}
		$\mathcal{Y}_{\pi}(\underline{n}-T)=\begin{cases}
		cr_n(UD_nF)(X_{\pi(1)}, ..., X_{\pi(n)} ) & \mbox{ if } \pi(\underline{n}) \subseteq T \\ 0 & \mbox{ otherwise } 
		\end{cases}$
	\end{center}
	The morphism $(1)$ is equivalent to $
	\chi(\underline{n}-T)\overset{\simeq}{\longrightarrow}\underset{\pi:\underline{n}\rightarrow \underline{n} }{\prod}\mathcal{Y}_{\pi}(\underline{n}-T).$
	
	- If $\pi$ is not a permutation and then not surjective, we can find an element $s\notin 
	\pi(\underline{n}).$ 
	All the maps $\mathcal{Y}_{\pi}(\underline{n}-T)\longrightarrow \mathcal{Y}_{\pi}(\underline{n}-T\cup\{s\})$ are isomorphisms, so $\mathcal{Y}_{\pi}$ is cartesian.
	
	- If $\pi$ is a permutation, $thofib(\mathcal{Y}_{\pi})\cong \mathcal{Y}_{\pi}(\underline{n}) =cr_nUD_nF( X_{\pi(1)}, ..., X_{\pi(n)} ).$
	
	Therefore $ thofib(\chi)\overset{\simeq}{\longrightarrow}\underset{\pi\in \Sigma_n }{\prod}cr_n(UD_nF)(X_{\pi(1)}, ..., X_{\pi(n)} ).$ Thus 
	\begin{align*}
	thofib(\chi)_{h\Sigma_n}&\overset{\simeq}{\longrightarrow}(\underset{\pi\in \Sigma_n }{\prod}cr_n(UD_nF)(X_{\pi(1)}, ..., X_{\pi(n)} ))_{h\Sigma_n}\\
	& \overset{\simeq}{\longrightarrow}cr_n(UD_nF)(X_{1}, ..., X_{n} ).
	\end{align*}
	
	Now that we have showed that $J$ is a quasi-isomorphism, let us consider  the chain map $\alpha$ using Lemma \ref{map_alpha}: 
	\begin{center}
		$\alpha: L_n(UF)\circ \triangle\overset{p_n(L_nUF\circ \triangle)}{\longrightarrow} P_n(L_n(UF)\circ \triangle)\overset{\simeq}{\longrightarrow} (L_n(UP_nF))\circ \triangle$
	\end{center}
	which is a weak equivalence since $L_n(UF)\circ \triangle$ is $n$-excisive. Putting all these together we form the following diagram:
	\begin{center}
		$\xymatrix{
			\underset{=  (\triangle_n(UF)(X)) _{h\Sigma_n} }{	(L_n(UF)\circ \triangle) _{h\Sigma_n}\ar[r]^-{\alpha }_-{\simeq} } &((L_n(UP_nF))\circ \triangle)_{h\Sigma_n} \\
			&(cr_n(UP_nF)\circ \triangle)_{h\Sigma_n}  \ar[u]^{p_1...p_1cr_nP_nF }_{\simeq} 
			&(cr_n(UD_nF)\circ \triangle)_{h\Sigma_n}  \ar[l]^{ }_{\simeq}\ar[r]^-{J}_-{\simeq}& UD_nF			
		} 
		$        
	\end{center}
	Note that $(\triangle (UF)(X))_{h\Sigma_n} \simeq red_0 (\widehat{\triangle} (UF(X)))_{h\Sigma_n}.$ One then deduce the following specializations for $\mathcal{D}:$
	
	\begin{enumerate}
		\item [-] When $\mathcal{D}= Ch_+,$ we have : 	$red_0( \widehat{\triangle} F(X)_{h\Sigma_n}) \simeq  D_nF(X) $ 
		\item[-] 	When $\mathcal{D}= \text{Alg}_\mathcal{O},$ 
		we apply the functor $\Omega^{\infty}(-)$ to the  above diagram and since  $\Omega^{\infty}(UD_nF)\simeq D_nF $ (by Remark \ref{trivial_algebra_structure}), we get the weak equivalence  	\begin{center}
			$\Omega^{\infty}( \widehat{\triangle} F(X)_{h\Sigma_n}) \simeq  D_nF(X) $
		\end{center}
		\item[-] When $\mathcal{D}=Ch,$ the diagram itself gives the proof.		
	\end{enumerate}
	
\end{Proof of theorem}

\begin{corollary}\label{Homogeneous_Infinite} Let $\mathcal{D}=\text{Alg}_\mathcal{O},$  $Ch_+$ or $Ch,$  and $F:\text{Alg}_\mathcal{O} \longrightarrow \mathcal{D}$  be a homotopy and reduced functor. Then there is a weak equivalence 
	\begin{center}
		$D_nF(X)\simeq \Omega^{\infty} H(\Sigma^{\infty}X)$
	\end{center}
	where $H: Ch_+\longrightarrow Ch_{\Bbbk}$  is the $n$-homogeneous functor given by: $H(V):=\widehat{\triangle}_n(F \mathcal{O}(-))(V)_{h\Sigma_n}.$	
	In particular when $\mathcal{D}=Ch_+$ or $Ch$ then this result holds when the ground field $\Bbbk$ is of any characteristics.
\end{corollary}
	\begin{proof}
	The functor $H$ is $n$-homogeneous since it is the $n$-th stabilization of the cross effect of  $F \mathcal{O}(-).$ 
	
	Let $X$ be an algebra over the operad $\mathcal{O},$ and $F:\text{Alg}_\mathcal{O}\longrightarrow \mathcal{D}$ be a homotopy and reduced functor.
	We observe that
	\begin{align*}
	\widehat{\triangle}_nF(X)&\simeq  \Omega^{n} (\widehat{\triangle}_nF)(\Sigma X) & (\text{ since  } L_nUF \text{ is n-multilinear })\\
	& \simeq \Omega^{n} (\widehat{\triangle}_nF)(\mathcal{O} (s\Sigma^{\infty}X)) & ( \text{ since } \Sigma X\simeq  \mathcal{O}(s\Sigma^{\infty}X ) \text{ from Corollary  }\ref{Suspension_algebra_is_free} )\\
	& \cong \Omega^{n} \widehat{\triangle}_n(F\mathcal{O}(-) )(s\Sigma^{\infty}X) & ( \text{ since } \mathcal{O}(-) \text{ commutes with coproducts})\\
	& \simeq  \widehat{\triangle}_n(F\mathcal{O}(-) )(\Sigma^{\infty}X) & (\text{ since  } L_n(UF\mathcal{O}(-)) \text{ is n-multilinear })
	\end{align*}
	One deduce from this observation that $ (\widehat{\triangle}_nF(X))_{ h\Sigma_n}\simeq   \widehat{\triangle}_n(F \mathcal{O}(-))(\Sigma^{\infty}X)_{h\Sigma_n}.$ 
	Using Theorem $\ref{Nick_Kuhn},$  we obtain the quasi-isomorphism \begin{center}
		$D_nF(X)\simeq \Omega^{\infty}(\widehat{\triangle}_n(F \mathcal{O}(-))(\Sigma^{\infty}X)_{h\Sigma_n}).$
	\end{center}
	
\end{proof}

\begin{corollary}\label{Derivatives}	Let $\mathcal{C}$ and $\mathcal{D}$ be any of the categories  $\text{Alg}_\mathcal{O} ,$  $ Ch_+$ or $Ch,$ and
	$F: \mathcal{C}\longrightarrow \mathcal{D}$  be a homotopy and reduced functor.
	If either $F$  is finitary  or if $X$ is finite we have the following quasi isomorphisms: 
	\begin{enumerate}
		\item If $\mathcal{C}= \text{Alg}_\mathcal{O},$ then there is a weak equivalence 
		\begin{center}
			$D_nF(X)\simeq \Omega^{\infty}(\widehat{\triangle}_nF(\mathcal{O}(\Bbbk))\otimes (\Sigma^{\infty}X)^{\otimes n})_{h\Sigma_n};$
		\end{center}
		
		\item If $\mathcal{C}=Ch_+$ or $Ch,$ then there is weak equivalence 
		\begin{center}
			$D_nF(X)\simeq  \Omega^{\infty}(  \widehat{\triangle}_nF(\Bbbk)\otimes (\Sigma^{\infty}X)^{\otimes n})_{h\Sigma_n}.$ 
		\end{center} 		
	\end{enumerate}
	In particular when $\mathcal{D}=Ch_+$ or $Ch$ then this result holds when the ground field $\Bbbk$ is of any characteristics.
\end{corollary}
To prove this result which classifies homogeneous functors in our algebraic point of view, we will need the following lemma.
\begin{mylemma}\label{Multilinear_assembly_map}
	Let $L_r: (Ch_+)^{\times r} \longrightarrow Ch$ be a $r$-reduced-multilinear functor. Then for any chain complexes $V_1, ..., V_r$ and finite chain complexes $W_1, ..., W_r,$ there is a zig-zag of quasi-isomorphisms
	\begin{center}
		$W_1\otimes ...\otimes W_r \otimes L_r(V_1, ..., V_r)  \simeq L_r(W_1\otimes V_1, ..., W_r\otimes V_r).$
	\end{center}
\end{mylemma}
\begin{proof}
	\begin{enumerate}
		\item We first consider the case $r=1$ and we want to construct a zig-zag of  quasi-isomorphisms $W\otimes L_1(V)\simeq L_1(W\otimes V),$ for a given chain complex $V$ and a finite one $W.$
		
		Let us consider the following commutative diagram
		\begin{center}
			$\xymatrix{
				L_1(sV\oplus V) \ar[d]& 0 \ar[l]_-{\simeq} \ar[r]^-{\simeq}\ar[d] &L_1(sV)\oplus s^{-1}L_1(sV)\ar[d] \\
				L_1(sV) & L_1(sV) \ar[l]_{=} \ar[r]^{=} &L_1(sV)\\
				L_1(0) \ar[u]& 0 \ar[l]_{\simeq} \ar[r]^{=} \ar[u]&\ar[u] 0\ar[u]  
			}$
		\end{center}
		A homotopy limit functor applied on each column gives the zig-zag of quasi-isomorphisms	
		\begin{align*}
		L_1(V) &\overset{\simeq}{\longleftarrow} \bullet \overset{\simeq}{\longrightarrow }s^{-1}L_1(sV) 
		\end{align*}
		where the 	homotopy limit result of the first column in due to the fact that the functor $L_1$ is linear.
		This later zig-zag can also be re-written as:
		\begin{align*}
		sL_1(V) &\overset{\simeq}{\longleftarrow} \bullet \overset{\simeq}{\longrightarrow }L_1(sV) 
		\end{align*}
		and thus equivalent to: $	\Bbbk u\otimes L_1(V) \overset{\simeq}{\longleftarrow} \bullet \overset{\simeq}{\longrightarrow }L_1(	\Bbbk u\otimes V),$
								
		for a given homogeneous element $u$ of degree $1.$
		One deduce inductively from this construction that $, \forall n\geq 0,$ we have 	$	(	\Bbbk u)^n \otimes L_1(V) \overset{\simeq}{\longleftarrow} \bullet \overset{\simeq}{\longrightarrow }L_1(	(	\Bbbk u)^n \otimes V) $	
		and therefore , given  any homogeneous element $u$ of arbitrary degree, we have a zig-zag of quasi-isomorphisms:  $\alpha_u: 	\Bbbk u\otimes L_1(V) \overset{\simeq}{\longleftarrow} \bullet \overset{\simeq}{\longrightarrow }L_1(	\Bbbk u\otimes V).$
	
		If $W=(\Bbbk u \oplus \Bbbk v, d)$ is a  chain complex with 2 generators, we set $\alpha_u+ \alpha_v$ to be the composition 	
		
		$\alpha_u+ \alpha_v: \xymatrix{ W\otimes L_1(V) & \ar[l]_-{\simeq}   \bullet  \ar[r]^-{\simeq}  & L_1(\Bbbk u\otimes V)\oplus L_1(\Bbbk v \otimes V)  \ar[rr]^-{\simeq} & & L_1(W\otimes V)
		},$
		where the last quasi-isomorphism is due to the fact that $L_1$ is linear.
		We generalize this construction inductively on the number of generators to any arbitrary finite chain complex $W.$	
		\item In the case that $r=2,$ let  $L_{2, V_1}$ be the linear functor $V_2 \longmapsto L_2(V_1, V_2);$ One have:
		\begin{align*}
		W_1\otimes W_2 \otimes L_2(V_1, V_2) \overset{\cong}{\longrightarrow}&	W_1\otimes (W_2 \otimes L_{2, V_1}( V_2))\\
		\overset{\simeq}{\longleftarrow}  \bullet 	\overset{\simeq}{\longrightarrow} &  W_1\otimes L_{2, V_1}(W_2\otimes V_2)\\
		\overset{\cong}{\longrightarrow} & W_1\otimes L_{2, W_2\otimes V_2 }(V_1)\\
		\overset{\simeq}{\longleftarrow}  \bullet 	\overset{\simeq}{\longrightarrow} & L_{2, W_2\otimes V_2 }(W_1\otimes V_1)= L_2(W_1\otimes V_1, W_2\otimes V_2)
		\end{align*}
		Again, we generalize this argument inductively to any arbitrary $r.$		
	\end{enumerate}
\end{proof}

\begin{Proof of corollary}
	
	We give the proof using the result of Theorem \ref{Nick_Kuhn} and Corollary \ref{Homogeneous_Infinite}. Namely,  let	$F: \mathcal{C}\longrightarrow \mathcal{D}$  be a homotopy and reduced functor  with
	$\mathcal{C}$ and $\mathcal{D}$ be any of the categories  $\text{Alg}_\mathcal{O} $ or $ Ch_+.$ Then  $D_nF(X)\simeq (\Omega^{\infty})red_0 H (\Sigma^{\infty}X),$
	where $H:Ch_+\longrightarrow Ch$ is a special case of a $n$-multilinear functor $(V_1, ..., V_n)\longmapsto H(V_1, ..., V_n).$
	\begin{enumerate}
		\item When $\mathcal{C}=\mathcal{D}= \text{Alg}_\mathcal{O},$  $H(V)=\widehat{\triangle}_n(F \mathcal{O}(-))(V)_{h\Sigma_n}=L_n(F \mathcal{O}(-))(V, ..., V)_{h\Sigma_n}.$
		
		If  $ \Sigma^{\infty}X$ is finite, then using Lemma \ref{Multilinear_assembly_map}, we have a $\Sigma_n$-equivariant zig-zag of  quasi-isomorphisms:
		\begin{center}
			$(\Sigma^{\infty}X)^{\otimes n} \otimes \widehat{\triangle}_n(F \mathcal{O}(-))(\Bbbk) 	\overset{\simeq}{\longleftarrow}  \bullet 	\overset{\simeq}{\longrightarrow} \widehat{\triangle}_n(F \mathcal{O}(-))(\Sigma^{\infty}X)$
		\end{center} 
		and therefore we deduce the quasi-isomorphism 
		\begin{center}
			$(\Sigma^{\infty}X)^{\otimes n} \underset{_{h\Sigma_n}}{\otimes }\widehat{\triangle}_n(F \mathcal{O}(-))(\Bbbk) 	\overset{\simeq}{\longleftarrow}  \bullet 	\overset{\simeq}{\longrightarrow} \widehat{\triangle}_n(F \mathcal{O}(-))(\Sigma^{\infty}X)_{h\Sigma_n}.$
		\end{center} 
		In addition, if $F$ is finitary then $L_n(F \mathcal{O}(-))$ is finitary on each variable. In this case for any arbitrary algebra $X,$ we rewrite  $\Sigma^{\infty}X$ as a filtered colimit of its finite subcomplexes and then apply again Lemma \ref{Multilinear_assembly_map} to these finite subcomplexes as above and recover a quasi-isomorphism 	
		\begin{center}
			$(\Sigma^{\infty}X)^{\otimes n} \underset{_{h\Sigma_n}}{\otimes }\widehat{\triangle}_n(F \mathcal{O}(-))(\Bbbk) 	\overset{\simeq}{\longleftarrow}  \bullet 	\overset{\simeq}{\longrightarrow} \widehat{\triangle}_n(F \mathcal{O}(-))(\Sigma^{\infty}X)_{h\Sigma_n}.$
		\end{center}
		\item	In all other cases of the categories $\mathcal{C}$ and $\mathcal{D,}$ we refer again to  Theorem \ref{Nick_Kuhn} and Corollary \ref{Homogeneous_Infinite} to chose the appropriate $H$ and follows an analogue road map as in $1.$ 
	\end{enumerate}

\end{Proof of corollary}

	\begin{definition}[Goodwillie derivatives]\label{definition_derivative}\index{$\partial_nF$ is the $n^{th}$ Goodwillie derivative.}
	Let  $\mathcal{D}$ be either $\text{Alg}_\mathcal{O} $ or $ Ch_+.$ 
	\begin{enumerate}		
		\item If $F: \text{Alg}_\mathcal{O}\longrightarrow \mathcal{D}$ is a homotopy and reduced functor,   then 
		$ \widehat{\triangle}_nF(\mathcal{O}(\Bbbk))$ is called the $n^{th}$ derivative (or $n^{th}$  Goodwillie derivative) of $F$ and is denoted $\partial_nF.$
		\item 	If $F:Ch_+\longrightarrow \mathcal{D} $ is a homotopy and reduced functor,   then 
		$ \widehat{\triangle}_nF(\Bbbk)$ is called the $n^{th}$ derivative (or $n^{th}$  Goodwillie derivative) of $F$ and is denoted $\partial_nF.$
	\end{enumerate}
	This definition holds when $\mathcal{D}=Ch_+$ or $Ch$ and the ground field $\Bbbk$ is of any characteristics.
\end{definition}

\section{Examples: Computing the Goodwillie derivatives}\label{Section Example computing Goodwillie derivatives}
In this section, we show how one can compute the Goodwillie derivatives for a couple of functors. 
\begin{example}
	The computation below shows that the Goodwillie derivatives of the identity functor $Id:\text{Alg}_\mathcal{O}  \longrightarrow \text{Alg}_\mathcal{O}  $ is given by: $\partial_*Id\simeq\mathcal{O}.$  
	\begin{align*}
	\partial_nId\simeq&\underset{p_i\rightarrow \infty}{\text{hocolim } } s^{-p_1-...-p_n}cr_nI(\mathcal{O}(\Sigma^{p_1}\Bbbk), ...,\mathcal{O}(\Sigma^{p_n} \Bbbk) )\\
	=&\underset{p_i\rightarrow \infty}{\text{hocolim }}  s^{-p_1-...-p_n} thofib(\underline{n}-T\mapsto   \mathcal{O}(\underset{i\in T}{\oplus}\Sigma^{p_i}\Bbbk) )\\
	\cong&\underset{p_i\rightarrow \infty}{\text{hocolim } } s^{-p_1-...-p_n}   \mathcal{O}(\underset{i\in \underline{n}}{\oplus}\Sigma^{p_i}\Bbbk) \\
	=&\underset{p_i\rightarrow \infty}{\text{hocolim } } s^{-p_1-...-p_n}  \underset{r\geq 0}{\bigoplus}  (\mathcal{O}(r)\otimes_{\Sigma_r}(\underset{i\in \underline{n}}{\oplus}s^{p_i}\Bbbk)^{\otimes r}) \\
	=& \mathcal{O}(n)\otimes_{\Sigma_n}(\Bbbk)^{\otimes n}\cong \mathcal{O}(n).
	\end{align*}
	We use an analogue computation to obtain $\partial_*\Sigma^{\infty}\Omega^{\infty}\simeq B(\mathcal{O}).$ 
\end{example}

The next example is transformed into a lemma. Let $V$ be a finite non negatively graded chain complex. By finite, we mean of finite dimension in each degree and bounded above. We define the functor 
\begin{center} $ N\Bbbk Hom_{Ch_+}(V\otimes N\Bbbk \triangle^{\bullet}, -): Ch_+ \longrightarrow Ch_+$\end{center}
where, $N: sAb \longrightarrow Ch_+$ is the normalization functor and 
$\Bbbk Hom_{Ch_+}(V\otimes N\Bbbk \triangle^{\bullet}, W)$ denotes the free simplicial $\Bbbk$-vector space generated by the simplicial set $Hom_{Ch_+}(V\otimes N\Bbbk \triangle^{\bullet}, W);$

\begin{mylemma}\label{Lemma Computing the derivatives of Representable functor}
	Let $V\in Ch_+.$ Then we have the quasi-isomorphism (in $Ch$)
	\begin{center}
		$ \partial_n N\Bbbk Hom_{Ch_+}(V\otimes N\Bbbk \triangle^{\bullet}, -) \simeq \underline{hom}(V, \Bbbk)^{\otimes n}$
	\end{center}
\end{mylemma}

Before we give the proof of this quasi-isomorphism,  we remind the following fact which seem to be a classical construction:
Let $p\in \mathbb{N},$ $A$ be a simplicial $\Bbbk$-vector space and consider the following notations:
\begin{enumerate}
	\item [-] We write $A[p]$ to mean the simplicial $\Bbbk$-vector space  given levelwise by $A[p]_n:=A_n\otimes \Bbbk s^p.$
	\item[-] If $(X_\bullet, *)$ is a pointed simplicial set then $\widetilde{\Bbbk}X_\bullet:= \Bbbk X_\bullet/\Bbbk *$
\end{enumerate} 
$A[p]$ is a $p$-connected Kan complex (as any simplicial abelian group), thus the Hurewicz map $A[p]\overset{h}{\longrightarrow} \widetilde{\Bbbk}A[p],$ which is in fact the unit of the adjoint pair $\Bbbk(-): sVect_\Bbbk\rightleftarrows sSet: U ,$ is $2p$-connected. The Hurewicz theorem stated on this current form appears in \cite[Thm 3.7]{Goerss1999} for abelian groups.

In addition, considering the natural projection $l:\widetilde{\Bbbk}A[p]  \longrightarrow A[p], 
\underset{i}{\oplus} x_i  \longmapsto \underset{i}{\Sigma} x_i.$

since the composite $A[p] \overset{h}{\longrightarrow }\widetilde{\Bbbk}A[p] \overset{l }{\longrightarrow} A[p]$ is the identity on $ A[p]\backslash \{0\},$ we deduce that the map $l$ is also $2p$-connected. Therefore the map $\Omega^{p}\widetilde{\Bbbk}A[p] \overset{\Omega^{p}(l)}{\longrightarrow}     \Omega^{p}A[p]$
 is $p$-connected and the map 
\begin{center}
	$\underset{p\rightarrow \infty}{\text{hocolim }}\Omega^{p}\widetilde{\Bbbk}A[p] \overset{ }{\longrightarrow}   \underset{p\rightarrow \infty}{\text{hocolim }}  \Omega^{p}A[p]$
\end{center}
is a weak equivalence of simplicial abelian groups. Now using the fact the the functor $N$ is a left and right Quillen functor of the Dold Kan correspondence we deduce the quasi-isomorphism

\begin{equation}\label{Equation weak equivalence simplicial abelian group}
\underset{p\rightarrow \infty}{\text{hocolim }}\Omega^{p}N\widetilde{\Bbbk}A[p] \overset{ }{\longrightarrow}   \underset{p\rightarrow \infty}{\text{hocolim }}  \Omega^{p}NA[p]
\end{equation}


\begin{Proof of Lemma derivative representable functor}
	
	We use Lemma \ref{Lemma Cross-effect equal Co-Cross-effect} to obtain the quasi-isomorphism:
	\begin{center}
		$cr_n(N\Bbbk Hom_{Ch_+}(V\otimes N\Bbbk \triangle^{\bullet},  -))(W_1, ..., W_n)\simeq$\\ $\text{thocofib }(\underline{n} \supseteq T \mapsto N\Bbbk Hom_{Ch_+}(V\otimes N\Bbbk \triangle^{\bullet},  \underset{i\in T}{\oplus} W_{i}     ))$
	\end{center}	
	On the other hand the functors $N:sAb \longrightarrow Ch_+$ and $\Bbbk(-):sSet\longrightarrow sAb $ are left Quillen functors, we therefore have the equivalences 
	\begin{align*}
	\text{thocofib }( N\Bbbk Hom_{Ch_+}(V\otimes N\Bbbk \triangle^{\bullet},  \underset{i\in T}{\oplus} W_{i}     ))
	& \simeq  N\Bbbk  \text{ thocofib}( Hom_{Ch_+}(V\otimes N \Bbbk \triangle^{\bullet},  \underset{i\in T}{\oplus} W_{i}     ))\\
	&\simeq  N\Bbbk  \text{ thocofib}( \underset{i\in T}{\oplus}Hom_{Ch_+}(V\otimes N \Bbbk \triangle^{\bullet},  W_{i}     ))
	\end{align*}
	Since the maps in the $\underline{n}$-cube of pointed simplicial sets: $T \longmapsto \underset{i\in T}{\oplus}Hom_{Ch_+}(V\otimes N \Bbbk \triangle^{\bullet},  W_{i} )$
	are inclusions, the total  homotopy colimit is the strict total cofiber (tcofib), and computation shows (inductively) that 
	\begin{center}
		$\text{tcofib }( \underset{i\in T}{\oplus}Hom_{Ch_+}(V\otimes N \Bbbk \triangle^{\bullet},  W_{i} )) \cong $\\ $   N\Bbbk (Hom_{Ch_+}(V\otimes N \Bbbk \triangle^{\bullet}, W_{1}       ) \wedge ... \wedge  Hom_{Ch_+}(V\otimes N \Bbbk \triangle^{\bullet},  W_{n} )) \cong $\\
		$    N(\widetilde{\Bbbk} Hom_{Ch_+}(V\otimes N \Bbbk \triangle^{\bullet}, W_{1}       ) \otimes ... \otimes  \widetilde{\Bbbk} Hom_{Ch_+}(V\otimes N \Bbbk \triangle^{\bullet},  W_{n} ) )$
	\end{center}
	We then conclude the quasi-isomorphism:
	\begin{equation}\label{Equation Cross effect Representable decomposition}
	cr_n( \widetilde{Ch}_+(V,-))(W_1, ..., W_n)\simeq  N\widetilde{\Bbbk} Hom_{Ch_+}(V\otimes N \Bbbk \triangle^{\bullet}, W_{1}       ) \otimes ... \otimes  N\widetilde{\Bbbk} Hom_{Ch_+}(V\otimes N \Bbbk \triangle^{\bullet},  W_{n}        )
	\end{equation}
	

	If $V$ is bounded below degree $k,$  we have
	\begin{align*}
	Hom_{Ch_+}(V\otimes N \Bbbk \triangle^{\bullet}, s^{p+k}\Bbbk)&\cong  Hom_{Ch_+}( N \Bbbk \triangle^{\bullet}, \underline{hom}(V,s^{p+k}\Bbbk)) \\
	& \cong Hom_{Ch_+}( N \Bbbk \triangle^{\bullet}, \underline{hom}(V,\Bbbk)\otimes s^{p+k}\Bbbk)\\
	& \overset{(1)}{\underset{\simeq}{\longleftarrow}} Hom_{Ch_+}( N \Bbbk \triangle^{\bullet}, \underline{hom}(V,\Bbbk)\otimes s^{k}\Bbbk)[p]
	\end{align*}
	where the weak equivalence $(1)$ is given by the weak equivalence of simplicial vector spaces 
	\begin{center}
		$Hom_{Ch_+}( N \Bbbk \triangle^{\bullet}, \underline{hom}(V,\Bbbk)\otimes s^{k}\Bbbk)\otimes Hom_{Ch_+}( N \Bbbk \triangle^{\bullet},  s^{p}\Bbbk)\simeq$\\
		$Hom_{Ch_+}( N \Bbbk \triangle^{\bullet}, \underline{hom}(V,\Bbbk)\otimes s^{p+k}\Bbbk)$
	\end{center}
	defined in \cite[ (2.8), p 295]{SS03}.
	
	Now, when we replace $A$ in the map (\ref{Equation weak equivalence simplicial abelian group}) with $Hom_{Ch_+}( N \Bbbk \triangle^{\bullet}, \underline{hom}(V,\Bbbk)\otimes s^{k}\Bbbk)$ and compose it with $\Omega^k(-),$ we get the quasi-isomorphisms
\begin{align*}
 \underset{p\rightarrow \infty}{\text{hocolim }}\Omega^{p}N\widetilde{\Bbbk}Hom_{Ch_+}( N \Bbbk \triangle^{\bullet}, \underline{hom}(V,\Bbbk)\otimes s^{p}\Bbbk) \simeq\\ \underset{p\rightarrow \infty}{\text{hocolim }}\Omega^{p+k}N\widetilde{\Bbbk}Hom_{Ch_+}( N \Bbbk \triangle^{\bullet}, \underline{hom}(V,\Bbbk)\otimes s^{k}\Bbbk)[p]\simeq \\
 \underset{p\rightarrow \infty}{\text{hocolim }}  \Omega^{p+k}NHom_{Ch_+}( N \Bbbk \triangle^{\bullet}, \underline{hom}(V,\Bbbk)\otimes s^{k}\Bbbk)[p] \simeq \\
 \underset{p\rightarrow \infty}{\text{hocolim }}  \Omega^{p+k} \underline{hom}(V,\Bbbk)\otimes s^{p+k}\simeq   \underline{hom}(V,\Bbbk)
\end{align*}	
	Using this above equivalence, we consider the specific case	$W_i= s^{p_i}\Bbbk$	in  Equation (\ref{Equation Cross effect Representable decomposition}) and apply the functor $\underset{p_i \rightarrow \infty}{\text{hocolim} }$ to the left-hand and right-hand side of this same equation, we get
	the quasi-isomorphism $ \partial_n N\Bbbk Hom_{Ch_+}(V\otimes N\Bbbk \triangle^{\bullet}, -) \simeq \underline{hom}(V, \Bbbk)^{\otimes n}.$

\end{Proof of Lemma derivative representable functor}

\section{Chain rule on the composition  of two functors }\label{Section Chain rule on the composition}

This section intends to describe the chain rule property that have the Goodwillie derivatives when we compose two functors.

\begin{theorem}\cite[Thm 1.15]{Ching10}\label{Chain_rule_ching}
	Let $F,G: Ch\longrightarrow Ch$ be homotopy functors, and suppose that $F$ is finitary, then 
	\begin{center}
		$\partial_{*}(FG)\simeq \partial_{*}(F)\circ \partial_{*}(G)$
	\end{center}
\end{theorem}

\begin{corollary}\label{Chain_rule_Chain_Complexes}
	Let $\mathcal{C}$ or $\mathcal{D}$ be any of  the model categories  $\text{Alg}_\mathcal{O}, Ch,$ or $Ch_+,$ and 
	$F,G$ be  homotopy and reduced  functors:  $\mathcal{C}       \overset{G}{\longrightarrow} Ch \overset{F}{\longrightarrow} \mathcal{D},$  and suppose that $F$ is finitary, then 
	\begin{center}
		$\partial_{*}(FG)\simeq \partial_{*}(F)\circ \partial_{*}(G)$
	\end{center}
\end{corollary}

\begin{proof}
	We prove this result only when $\mathcal{C}=\text{Alg}_\mathcal{O}$ since the other cases are obtained from a slight adaptation and following the same idea. If $U: \text{Alg}_\mathcal{O}\longrightarrow Ch$ denotes the forgetful functor, then we make the following computation:
	\begin{align*}
	\partial_n(FG)\simeq &  \widehat{\triangle}_n(FG)( \mathcal{O}(\Bbbk))\\
	=	&\underset{p_i\rightarrow \infty}{ \text{ hocolim }} s^{-p_1-...-p_n}cr_n(FG)(\mathcal{O}(\Sigma^{p_1}\Bbbk), ...,\mathcal{O}(\Sigma^{p_n} \Bbbk) )\\
	\cong	& \underset{p_i\rightarrow \infty}{ \text{ hocolim }} s^{-p_1-...-p_n}cr_n(UFG)(\mathcal{O}(\Sigma^{p_1}\Bbbk), ...,\mathcal{O}(\Sigma^{p_n} \Bbbk) )\\
	\cong	& \underset{p_i\rightarrow \infty}{ \text{ hocolim }} s^{-p_1-...-p_n}cr_n(UFG\mathcal{O}(-))( \Sigma^{p_1}\Bbbk, ...,\Sigma^{p_n} \Bbbk )\\
	\simeq &  \underset{p_i\rightarrow \infty}{ \text{ hocolim } }s^{-p_1-...-p_n}cr_n(UFG \mathcal{O}(-)red_0 )(\Sigma^{p_1}\Bbbk, ...,\Sigma^{p_n} \Bbbk)\\
	\simeq &  \partial_n (UFG\mathcal{O}(-)red_0 )
	\end{align*}
	The two functors $Ch\overset{G\mathcal{O}(-)red_0}{\longrightarrow} Ch \overset{UF}{\longrightarrow}Ch$ are homotopy functors and  $UF$ preserves filtered homotopy colimits. Therefore, we use Theorem \ref{Chain_rule_ching} to claim that 
	\begin{align*}
	\partial_*(FG)\simeq & \partial_* (UF) \circ \partial_*(G \mathcal{O}(-)red_0 )\\
	\simeq & \partial_* (F) \circ \partial_*(G).
	\end{align*}
	
\end{proof}

				\bibliographystyle{alpha}
				\bibliography{mybibliography}

			\end{document}